\documentclass[a4paper]{amsart}

\usepackage{amsthm}
\usepackage{amssymb}
\usepackage[all]{xy}
\usepackage{hyperref}
\usepackage{verbatim}

\SelectTips{cm}{}
\calclayout 

\numberwithin{equation}{section}  

\theoremstyle{plain}
\newtheorem{lemma}[equation]{Lemma}
\newtheorem{theorem}[equation]{Theorem}
\newtheorem{proposition}[equation]{Proposition}
\newtheorem{corollary}[equation]{Corollary}
\newtheorem*{thm1}{Theorem~\ref{th:lgp}}
\newtheorem*{thm2}{Theorem~\ref{th:stratification}}

\theoremstyle{definition}
\newtheorem{definition}[equation]{Definition}
\newtheorem{example}[equation]{Example}

\theoremstyle{remark} 
\newtheorem{remark}[equation]{Remark} 

\newcommand{\Coh}{\operatorname{\mathsf{Coh}}}

\newcommand{\colim}{\operatorname{colim}}

\newcommand{\End}{\operatorname{End}}
\newcommand{\Ext}{\operatorname{Ext}}

\newcommand{\Hom}{\operatorname{Hom}}

\newcommand{\Id}{\operatorname{Id}}

\newcommand{\Ker}{\operatorname{Ker}}

\newcommand{\Loc}{\operatorname{\mathsf{Loc}}}

\renewcommand{\mod}{\operatorname{\mathsf{mod}}}
\newcommand{\Mod}{\operatorname{\mathsf{Mod}}}

\newcommand{\Spec}{\operatorname{Spec}}

\newcommand{\StMod}{\operatorname{\mathsf{StMod}}}
\newcommand{\stmod}{\operatorname{\mathsf{stmod}}}
\newcommand{\Supp}{\operatorname{Supp}}
\newcommand{\supp}{\operatorname{supp}}
\newcommand{\Thick}{\operatorname{\mathsf{Thick}}}

\newcommand{\op}{\mathrm{op}}

\newcommand{\per}{\mathrm{per}}

\newcommand{\Ab}{\mathsf{Ab}} 
\newcommand{\comp}{\mathop{\circ}}
\newcommand{\col}{\colon}

\newcommand{\kos}[2]{{#1}/\!\!/{#2}}

\newcommand{\lto}{\longrightarrow}
\newcommand{\Ra}{\Rightarrow}

\newcommand{\xra}{\xrightarrow}

\newcommand{\bik}{Benson/Iyengar/Krause}

\def\mcU{\mathcal{U}} 
\def\mcV{\mathcal{V}}
\def\mcW{\mathcal{W}} 
 
\def\mcZ{\mathcal{Z}}

\def\sfc{\mathsf c}

\def\sfA{\mathsf A} 
 
\def\sfC{\mathsf C}
\def\sfD{\mathsf D}

\def\sfS{\mathsf S} 
\def\sfT{\mathsf T} 
\def\sfU{\mathsf U}
\def\sfV{\mathsf V}

\def\bbN{\mathbb N}

\def\bbZ{\mathbb Z}

\newcommand{\fa}{\mathfrak{a}} 
\newcommand{\fb}{\mathfrak{b}}
 
\newcommand{\fp}{\mathfrak{p}}
\newcommand{\fq}{\mathfrak{q}}

\newcommand{\gam}{\varGamma}

\def\Si{\Sigma}

\def\one{\mathbf 1}

\title[A local-global principle for small triangulated categories]{A
  local-global principle\\ for small triangulated categories}

\author[Benson, Iyengar, Krause]{Dave Benson, Srikanth B. Iyengar, Henning Krause}

\address{Dave Benson \\ 
Institute of Mathematics\\ 
University of Aberdeen\\ 
King's College\\ 
Aberdeen AB24 3UE\\ 
Scotland U.K.}

\address{Srikanth B. Iyengar\\ 
Department of Mathematics\\
University of Nebraska\\ 
Lincoln, NE 68588\\ 
U.S.A.}

\address{Henning Krause\\ 
Fakult\"at f\"ur Mathematik\\ 
Universit\"at Bielefeld\\ 
33501 Bielefeld\\ 
Germany.}

\begin{document}

\begin{abstract}
  Local cohomology functors are constructed for the category of
  cohomological functors on an essentially small triangulated category
  $\sfT$ equipped with an action of a commutative noetherian
  ring. This is used to establish a local-global principle and to
  develop a notion of stratification, for $\sfT$ and the cohomological
  functors on it, analogous to such concepts for compactly generated
  triangulated categories.
\end{abstract}


\thanks{Version from July 8, 2014.\\
  The authors are grateful to MSRI at Berkeley for support while this
  work was in progress.  The second author was partly supported by NSF
  grant DMS 1201889 and a Simons Fellowship.}

\maketitle 
\setcounter{tocdepth}{1}
\tableofcontents

\section{Introduction}

In this paper we establish an analogue of the local-global principle
from commutative algebra \cite[Chap.~II, \S3]{Bourbaki:1961a} for an
essentially small triangulated category, using the central action of a
graded commutative ring. This has applications to the theory of
support varieties in representation theory and in commutative algebra,
that started with the work of Quillen \cite{Quillen:1971ab} and
Carlson \cite{Carlson:1983a}.

Our paradigms for the local-global principle are the ones from
\cite{Benson/Iyengar/Krause:2011a} and from Stevenson's
work~\cite{Stevenson:2011a} for compactly generated (so ``big'')
triangulated categories. These play a crucial role in the
classification of localising subcategories of the stable module
category of the group algebra of a finite
group~\cite{Benson/Iyengar/Krause:2011a} and of the singularity
category of a locally complete intersection
ring~\cite{Stevenson:2011b}. The local-global principle from
\cite{Benson/Iyengar/Krause:2011a} does yield an analogue (see
Proposition~\ref{pr:big-implies-small}) for the ``small'' category of
compact objects, which is useful in studying its thick subcategories,
for example. The work presented here arose from a search for a more
direct proof of this result. Our reason for doing so, besides the
obvious aesthetic one, is that there are small categories for which
there is no canonical choice of a big category. And even if there were
one, it is not clear that an action of a ring of operators on the
small category extends to an action on the big one.

Following an idea of Grothendieck--Verdier
\cite{Grothendieck/Verdier:1972a}, in this work we propose a different
model for such constructions. Namely, given an essentially small
triangulated category $\sfT$, we consider the category $\Coh\sfT$ of
cohomological functors $\sfT^\op\to\Ab$ into the category of abelian
groups.  Up to an equivalence, this is the category of ind-objects of
$\sfT$ in the sense of \cite{Grothendieck/Verdier:1972a}. The functor
category contains a copy of $\sfT$, because Yoneda's lemma allows to
identify an object $X\in\sfT$ with the representable functor
$\Hom_\sfT(-,X)$. While $\Coh\sfT$ is no longer triangulated, it does
carry an exact structure that is sufficient for our purposes;
moreover, it admits filtered colimits. These are the basic ingredients
we use for setting up our machinery.

Let us explain the main results in this work. Fix a noetherian graded
commutative ring $R$ acting centrally on $\sfT$; the principal
examples are listed in Example~\ref{ex:ring-actions}. This gives for each pair
of objects $X,Y$ in $\sfT$ an $R$-action on the graded abelian group
\[
\Hom^*_\sfT(X,Y)=\bigoplus_{n\in\bbZ}\Hom_\sfT(X,\Si^{n}Y).
\]
Let $\Spec R$ denote the set of homogeneous prime ideals. For each
$\fp\in\Spec R$ there is a localisation functor $\sfT\to \sfT_\fp$
taking an object $X$ to $X_\fp$. The category $\sfT_\fp$ has the same
objects as $\sfT$ and there is a natural isomorphism
\[\Hom^*_\sfT(X,Y)_\fp\xra{\sim}\Hom^*_{\sfT_\fp}(X_\fp,Y_\fp).\]

The formulation of the following local-global principle involves
\emph{Koszul objects}. Given an object $X\in\sfT$ and a homogeneous
ideal $\fa$ of $R$, an iterated cone construction yields an object
$\kos X\fa$. While this object depends on a choice of a sequence of
generators of $\fa$, the thick subcategory generated by it does not;
see Lemma~\ref{le:kos}. For $\fp\in\Spec R$ set
$X(\fp)=\kos{X_\fp}\fp$.

\begin{thm1}[Local-global principle]
 Let $\sfS$ be a thick subcategory of $\sfT$. Then the following
  conditions are equivalent for an object $X$ in $\sfT$:
\begin{enumerate}
\item $X$ belongs to $\sfS$.
\item $X_\fp$ belongs to $\Thick(\sfS_\fp)$ for each $\fp\in\Spec R$.
\item $X(\fp)$ belongs to $\Thick(\sfS_\fp)$ for each $\fp\in\Spec R$.
\end{enumerate}
\end{thm1}

Motivated by this result, we define the \emph{support} of an object
$X\in\sfT$ to be the set
\[
\supp_R X=\{\fp\in\Spec R\mid X(\fp)\neq 0\}.
\]
When the $R$-module $\End^*_\sfT(X)$ is finitely generated, this
coincides with the support, in the usual sense in commutative algebra,
of the $R$-module $\End^{*}_{\sfT}(X)$; see
Proposition~\ref{pr:tria_support}. For each $\fp$ in $\Spec R$, set
\[
\gam_\fp\sfT=\{X\in\sfT_\fp\mid \End_{\sfT_\fp}^*(X)_\fq=0\text{ for
  all }\fq\not\supseteq\fp\}.
\]
This is a thick subcategory of $\sfT_{\fp}$.  We say that $\sfT$ is
\emph{stratified} by the action of $R$ if for each $\fp$ in $\Spec R$,
the category $\gam_\fp\sfT$ admits no proper thick subcategory.

\begin{thm2}
  Suppose that $\sfT$ is stratified by the action of $R$. For any pair
  of objects $X,Y$ in $\sfT$ one has
\begin{align*}
X\in\Thick(Y) \quad&\iff\quad \supp_R X\subseteq\supp_R Y,\\
\Hom^*_\sfT(X,Y)=0\quad&\iff\quad (\supp_R X)\cap(\supp_R
Y)=\varnothing.
\end{align*}
\end{thm2}

There is a partial converse when the endomorphism rings of objects in
$\sfT$ are finitely generated $R$-modules: If $\sfT$ is not
stratified, then there are objects in $\sfT$ having the same support but
generating different thick subcategories; see
Proposition~\ref{pr:stratification-converse}.

The proofs of Theorems~\ref{th:lgp} and \ref{th:stratification}
involve the category of cohomological functors. Indeed, the $R$-action
on $\sfT$ has an obvious extension to $\Coh\sfT$, and based on this we
develop a theory of local cohomology and support for objects in
$\Coh\sfT$, analogous to the one in \cite{
  Benson/Iyengar/Krause:2008a} for compactly generated triangulated
categories. This forms the foundation for much of this work.

In order to illustrate these results and techniques, it is shown that
the category of perfect complexes over a commutative noetherian ring
is stratified; this amounts to a classical theorem of Hopkins
\cite{Hopkins:1987a} and Neeman \cite{Neeman:1992b}. Applications of
the local-global principle, Theorem~\ref{th:lgp}, to the study of
modules over locally complete intersections and over integral group
rings, will appear elsewhere.

Most ideas in this paper are taken from our previous work
\cite{Benson/Iyengar/Krause:2008a, Benson/Iyengar/Krause:2011a,
  Benson/Iyengar/Krause:2011b}; see also the references given there
for inspiration by other authors. However, the categorical setting in
this work is fundamentally different and the systematic use of the
category of cohomological functors in this context seems to be new.
We do not work with a compactly generated triangulated category having
arbitrary coproducts, but instead with an essentially small
triangulated category. A brief comparison between these two approaches
can be found in the final section. Otherwise, references to previous
work are kept to a minimum.

\section{Cohomological functors}

In this section we introduce the category of cohomological functors on
a triangulated category and study its basic properties. For instance,
we discuss base change and a long exact sequence corresponding to a
Verdier quotient.

\subsection*{Cohomological functors}

Let $\sfT$ be an essentially small triangulated category with
suspension $\Si\colon\sfT\xra{\sim}\sfT$.  Recall that a functor
$\sfT^\op\to\Ab$ into the category of abelian groups is
\emph{cohomological} if it takes exact triangles to exact sequences.
We denote by $\Coh\sfT$ the category of cohomological
functors. Morphisms in $\Coh\sfT$ are natural transformation and the
Yoneda functor $\sfT\to\Coh\sfT$ sending $X\in\sfT$ to
\[H_X=\Hom_\sfT(-,X)\] is fully faithful.  The suspension $\Si$ extends
to a functor $\Coh\sfT\xra{\sim}\Coh\sfT$ by taking $F$ in $\Coh\sfT$
to $F\comp\Si^{-1}$; we denote this again by $\Si$.  

It is convenient to view $\Coh\sfT$ as a full subcategory of the
category $\Mod\sfT$ of all additive functors $\sfT^\op\to\Ab$.  Note
that (co)limits in $\Mod\sfT$ are computed pointwise.  For $E$ and $F$
in $\Mod\sfT$ we write $\Hom(E,F)$ for the set of morphisms from $E$
to $F$.  Thus $\Hom(H_X,F)\cong F(X)$ for $X$ in $\sfT$, by Yoneda's
lemma.

Any additive functor $F\colon\sfT^\op\to\Ab$ can be written
canonically as a colimit of representable functors
\begin{equation}\label{eq:slice}
(\colim_{H_X\to F} H_X) \xra{\sim}F
\end{equation}
where the colimit is taken over the slice category $\sfT/F$; see
\cite[Proposition~3.4]{Grothendieck/Verdier:1972a}. Objects in
$\sfT/F$ are morphisms $ H_X\to F$ where $X$ runs through the objects
of $\sfT$. A morphism in $\sfT/F$ from $H_X\xra{\phi} F$ to
$H_{X'}\xra{\phi'} F$ is a morphism $\alpha \colon X\to X'$ in $\sfT$
such that $\phi' H_{\alpha}=\phi$.

A theorem of Lazard says that a module is flat if and only if it is a
filtered colimit of finitely generated free modules; this has been
generalised to functor categories by Oberst and R\"ohrl
\cite{Oberst/Roehrl:1970a}.  The following lemma shows that
cohomological and flat functors agree; this is well-known, for
instance from \cite[Lemma~2.1]{Neeman:1992a}.

\begin{lemma}
\label{le:flat}
The cohomological functors $\sfT^\op\to\Ab$ are precisely the filtered
colimits of representable functors (in the category of additive
functors $\sfT^\op\to\Ab$).  In particular, the category $\Coh\sfT$
has filtered colimits.
\end{lemma}

\begin{proof}[Sketch of proof]
  A filtered colimit of exact sequences is again exact. Thus a
  filtered colimit of representable functors is
  cohomological. Conversely, given a cohomological functor $F$, one
  easily checks that the slice category $\sfT/F$ is filtered. Thus
  \eqref{eq:slice} gives a presentation of $F$ as a filtered colimit
  of representable functors.
\end{proof}

We say that a sequence of morphisms in $\Coh\sfT$ is \emph{exact}
provided that evaluation at each object in $\sfT$  yields an
exact sequence in $\Ab$.

\begin{lemma}\label{le:quillen}
  The category $\Coh\sfT$ is an exact category in the sense of
  Quillen; it admits enough projective and enough injective objects.
\end{lemma}
\begin{proof}
  The cohomological functors form an extension closed subcategory of
  $\Mod\sfT$ containing all projective objects and all injective
  objects. This is clear for the projective objects and follows easily
  from Yoneda's lemma for the injectives.
\end{proof}

\subsection*{Exact functors}

Let $\sfT$ and $\sfU$ be essentially small triangulated categories. A
functor $P\colon\Coh\sfT\to\Coh\sfU$ is said to be \emph{exact} if it
takes exact sequences to exact sequences and if there is a natural
isomorphism $P\comp \Si\xra{\sim}\Si\comp P$.

An exact functor $f\colon\sfT\to\sfU$ induces a pair of functors
\[
f^*\colon\Coh\sfT\lto\Coh\sfU\qquad\text{and}\qquad
f_*\colon\Coh\sfU\lto\Coh\sfT
\] 
where $f^*(F)=\colim_{H_X\to F} H_{f(X)}$ and $f_*(G)=G\comp f$. The
next lemma collects some of their basic properties. Recall that a
\emph{triangulated subcategory} is a full additive subcategory closed
under forming cones and suspensions. The functor $f$ is a
\emph{quotient functor} when it is equivalent to the canonical functor
$\sfT\to\sfT/\sfS$ given by a triangulated subcategory
$\sfS\subseteq\sfT$.

\begin{lemma}
\label{le:exact-fun}
Let $f\colon\sfT\to\sfU$ be an exact functor between essentially small
triangulated categories.
\begin{enumerate} 
\item The functor $f^*$ is a left adjoint of $f_*$.
\item The functors $f^*$ and $f_*$ are exact and preserve filtered
  colimits.
\item If $f$ is fully faithful, then $f^*$ is fully faithful and
  $\Id\xra{\sim}f_*\comp f^*$.
\item If $f$ is a quotient functor, then $f_*$ is fully faithful and
  $f^*\comp f_*\xra{\sim}\Id$.
\end{enumerate} 
\end{lemma}
\begin{proof}
  (1) Given $F\in\Coh\sfT$ and $G\in\Coh\sfU$, we claim that
\[\Hom(f^*F,G)\cong\Hom(F,f_*G).\]
When $F$ is representable this is immediate from Yoneda's lemma. The
general case then follows since $F$ can be written as a colimit of
representable functors.

(2) Clearly, $f_*$ is exact and preserves filtered colimits. A left
adjoint, in particular, $f^{*}$, automatically preserves colimits. The
exactness of $f^*$ follows from the fact that $f$ is exact; see
\cite[Lemma 2.2]{Krause:2000a}.

(3) We use the fact that for any pair $(S,T)$ of adjoint functors, the
left adjoint $S$ is fully faithful iff the unit $\Id\to T\comp S$ is
invertible; see \cite[Proposition~1.3]{Gabriel/Zisman:1967a}. If $f$
is fully faithful, then $F\cong (f_*\comp f^*)( F)$ for any
representable functor $F$, and the general case follows by taking
filtered colimits.

(4) If $f$ is a quotient functor, then $f_*$ is fully faithful; see
\cite[Lemma~1.2]{Gabriel/Zisman:1967a}. Thus the counit $f^*\comp
f_*\to\Id$ is invertible, by the argument dual to the one in (3).
\end{proof}

\begin{remark}
  Observe that when $f^*$ is fully faithful so is $f$, since the
  Yoneda embedding of $\sfT$ into $\Coh\sfT$ is fully
  faithful. However, there are examples where $f_*$ is fully faithful
  but $f$ is not a quotient functor; see
  \cite[Example~7.4]{Krause:2005b}.
\end{remark}

\subsection*{Triangulated subcategories}
Let $\sfS\subseteq \sfT$ be a triangulated subcategory. We write
$i\colon\sfS\to\sfT$ for the inclusion and $q\colon\sfT\to\sfT/\sfS$
for the corresponding quotient functor.  Henceforth we view $\Coh\sfS$
and $\Coh\sfT/\sfS$ as full subcategories of $\Coh\sfT$, via $i^*$ and
$q_*$ respectively. More specifically, there are identifications
\[
\begin{aligned}
  \Coh\sfS&= \{F\in\Coh\sfT\mid F=\colim_\alpha H_{X_\alpha}\text{ with all } X_\alpha\in\sfS\},\\
  \Coh\sfT/\sfS&= \{F\in\Coh\sfT\mid F|_\sfS=0\}.
\end{aligned}
\]
The second identification follows from the universal property of the
quotient $\sfT/\sfS$ \cite[Chap.~II, Cor.~2.2.11]{Verdier:1997a}.
Here is a useful recognition criterion for objects in $\Coh\sfS$.

\begin{lemma}
\label{le:cofinal}
Let $\sfS\subseteq\sfT$ be a triangulated subcategory. Then
$F\in\Coh\sfT$ belongs to $\Coh\sfS$ iff each morphism $H_X\to F$ with
$X\in\sfT$ factors through $H_Y$ for some $Y\in\sfS$.
\end{lemma}

\begin{proof}
As in \eqref{eq:slice}, we write $F$ as a filtered colimit 
\[
F=\colim_{H_X\to F}H_X.
\] 
Then the assertion is an immediate consequence of the following lemma,
applied with $\sfC'$ and $\sfC$ the slice categories $\sfS/F$ and
$\sfT/F$, respectively.
\end{proof}

\begin{lemma}
  Let $i\colon \sfC'\to \sfC$ be a fully faithful functor with $\sfC$
  a small filtered category.  Suppose that for any $X\in\sfC$ there is
  an object $Y\in\sfC'$ and a morphism $X\to iY$. Then $\sfC'$ is a
  small filtered category, and for any functor $F\colon \sfC\to\sfD$
  into a category which admits filtered colimits, the natural morphism
\[
\colim_{Y\in\sfC'}F(iY)\lto \colim_{X\in\sfC}F(X)
\] 
is an isomorphism.
\end{lemma}

\begin{proof}
See \cite[Proposition~8.1.3]{Grothendieck/Verdier:1972a}.
\end{proof}

\subsection*{Localisation}
Let $\sfS\subseteq\sfT$ be a triangulated subcategory. With $i\col
\sfS\to \sfT$ and $q\col \sfT\to \sfT/\sfS$ the canonical functors,
set
\begin{equation}
\label{eq:loc-fun}
\gam=i^*\comp i_*\qquad\text{and}\qquad L=q_*\comp q^*.
\end{equation} 
These are exact functors on $\Coh \sfT$. By definition, for any $F\in \Coh\sfT$ one has
\begin{equation}
\label{eq:gam-F}
\gam F =  \colim_{H_S\to F} \Hom_{\sfT}(-,S)
\end{equation}
where the colimit is taken over the slice category $\sfS/F$, which is
filtered because $\sfS$ is a triangulated subcategory.

\begin{proposition}
\label{pr:loc-seq}
In $\Coh\sfT$ each object $F$ fits into a functorial exact sequence 
\begin{equation}\label{eq:loc-seq}
 \cdots \lto\Si^{-1} (L F)\lto \gam F\lto F\lto LF\lto \Si (\gam F)\lto
 \Si F\lto\cdots 
\end{equation}
Moreover, any exact sequence 
\begin{equation}\label{eq:loc-seq1}
 \cdots \lto \Si^{-1}F''\lto F'\lto F\lto F''\lto \Si F'\lto \Si F\lto\cdots 
\end{equation}
in $\Coh\sfT$ with $F'\in\Coh\sfS$ and $F''\in\Coh\sfT/\sfS$ is
isomorphic to \eqref{eq:loc-seq}.

For $F=\Hom_\sfT(-,X)$, the sequence \eqref{eq:loc-seq} specialises  to  
\begin{multline}
\label{eq:loc-long} 
\cdots\lto \Hom_{\sfT/\sfS}(-,\Si^{-1}X)\lto \colim_{S\to
  X}\Hom_\sfT(-,S)\lto \\ \lto\Hom_\sfT(-,X)\lto
\Hom_{\sfT/\sfS}(-,X)\lto \cdots
\end{multline}
\end{proposition}

\begin{proof} 
  Fix an object $X$ in $\sfT$. Specialising \eqref{eq:gam-F} one gets
  that
\[
\gam\Hom_\sfT(-,X) = \colim_{S\to X}\Hom_\sfT(-,S)
\] 
This functor fits into an exact sequence of cohomological functors of
the form \eqref{eq:loc-long} since one has by definition
\[
\Hom_{\sfT/\sfS}(-,X)=\colim_{X\to Y}\Hom_\sfT(-,Y)
\]
where $X\to Y$ runs through all morphisms with cone in $\sfS$.  The
sequence is functorial and hence yields an exact sequence for each
filtered colimit of representable functors. This justifies
\eqref{eq:loc-seq}.

Given another sequence \eqref{eq:loc-seq1}, we apply the exact functor
$L$ and obtain the following
commuting diagram.
\[
\xymatrix{
 \cdots \ar[r]& F' \ar[r]\ar[d]& F
 \ar[r]\ar[d]& F'' \ar[r]\ar[d]& \Si F' \ar[r]\ar[d]& \cdots \\
\cdots \ar[r]& LF' \ar[r]& LF
 \ar[r]& LF'' \ar[r]& L(\Si F') \ar[r]&\cdots 
}\]
The morphism $F\to F''$ is isomorphic to $F\to LF$,
since $LF'=0=\Si(LF')$ and $F''\xra{\sim}LF''$. Analogously, an
application of $\gam$ shows that $F'\to F$ is isomorphic to $\gam F\to
F$. This yields an isomorphism between \eqref{eq:loc-seq} and
\eqref{eq:loc-seq1}.
\end{proof}

The functor $L$ from \eqref{eq:loc-fun} is a \emph{localisation
  functor}\footnote{The natural morphism $\eta\colon \Id\to L$ has the
  property that $L\eta$ is invertible and $L\eta=\eta L$.}, while
$\gam$ is a \emph{colocalisation functor}\footnote{The natural
  morphism $\theta\colon \gam\to \Id$ has the property that
  $\gam\theta$ is invertible and $\gam\theta=\theta \gam$.}, and the
functorial exact sequence \eqref{eq:loc-seq} is called the
\emph{localisation sequence} for $\sfS\subseteq\sfT$. In the following
we consider the case that one of the natural morphisms $\gam F\to F$
and $F\to LF$ is an isomorphism.

\begin{corollary}\label{co:loc-seq}
\pushQED{\qed}   
Let $\sfS\subseteq\sfT$ be a triangulated subcategory and $F\in\Coh\sfT$. Then
\begin{align*}
F\in\Coh\sfS\quad&\iff \quad \gam F\cong F\quad \iff\quad LF=0, \\
F\in\Coh\sfT/\sfS\quad &\iff \quad F\cong LF\quad \iff\quad \gam F=0. \qedhere
\end{align*} 
\end{corollary}

A pair $(\sf U,\sfV)$ of full subcategories of an additive category forms a
\emph{torsion pair} provided that the inclusion of $\sfU$ admits a
right adjoint, the inclusion of $\sfV$ admits a left adjoint,
$\sfU=\{X\mid \Hom(X,Y)=0\text{ for all }Y\in\sfV\}$, and $\sfV=\{Y\mid
\Hom(X,Y)=0\text{ for all }X\in\sfU\}$.
 
\begin{corollary}
\label{co:torsion}
Let $\sfS\subseteq\sfT$ be a triangulated subcategory. Then $\Coh\sfS$
and $\Coh\sfT/\sfS$ form a torsion pair in $\Coh\sfT$. Thus
\begin{gather*}
  \Coh\sfT/\sfS=\{F\in\Coh\sfT\mid\Hom(E,F)=0\text{ for all
  }E\in\Coh\sfS\},\\
\Coh\sfS=\{F\in\Coh\sfT\mid\Hom(F,G)=0\text{ for all }G\in\Coh\sfT/\sfS\}.
\end{gather*}
\end{corollary}

\begin{proof}
  We have already seen in Lemma~\ref{le:exact-fun} that the inclusion
  of $\Coh\sfS$ admits a right adjoint while the inclusion of
  $\Coh\sfT/\sfS$ admits a left adjoint. For the first equality,
  observe that $\Hom(E,F)=0$ for all $E\in\Coh\sfS$ means that
  $F(X)=0$ for all $X\in\sfS$. This is equivalent to $\gam F=0$, and
  therefore to $F\in \Coh\sfT/\sfS$. For the second equality, it
  remains to show that $\Hom(F,G)=0$ for all $G\in\Coh\sfT/\sfS$
  implies $F\in\Coh\sfS$. The assumption on $F$ implies $LF=0$ since
  $L=q_*\comp q^*$ and \[\Hom(q^*F,q^*F)\cong\Hom(F,(q_*\comp
  q^*)F)=0.\] Thus $F$ belongs to $\Coh\sfS$.
\end{proof}

In general, (co)localisation functors do note commute; see
\cite[Example~3.5]{\bik:2008a}. The following lemma identifies some
conditions under which they do.

\begin{lemma}
\label{le:rules-commute}
Let $\sfS_1\subseteq \sfS_2\subseteq\sfT$ be triangulated
subcategories and let $(\gam_1,L_1)$ and $(\gam_2,L_2)$ be the
corresponding pairs of (co)localisation functors on $\Coh\sfT$. The
morphisms in \eqref{eq:loc-seq} induce isomorphisms
\[
\gam_1\gam_2\cong\gam_1\cong\gam_2\gam_1,\quad L_1L_2\cong L_2\cong
L_2 L_1,\quad \gam_1 L_2=0=L_2\gam_1,\quad \gam_2 L_1\cong L_1\gam_2.
\]
\end{lemma}

\begin{proof}
Apply the localisation sequence \eqref{eq:loc-seq}.
\end{proof}

This has the following useful consequence.

\begin{corollary}
\label{cor:detecting-subcategories}
\pushQED{\qed} 
Given thick subcategories $\sfS_1$ and $\sfS_2$ of $\sfT$, one has
\[
\sfS_1\subseteq\sfS_2\qquad\iff\qquad\Coh\sfS_1\subseteq\Coh\sfS_2.\qedhere
\]
\end{corollary}

\section{Cohomological localisation}

In this section we introduce cohomological localisation functors for
categories of cohomological functors and explain how to compute these
functors in terms of Koszul objects. These are analogues of results in
\cite[\S\S4--6]{\bik:2008a}.

Let $\sfT$ be an essentially small triangulated category. For objects $X,Y$ in $\sfT$ set
\[
\Hom^*_\sfT(X,Y)=\bigoplus_{n\in\bbZ}\Hom_\sfT(\Si^{-n}X, Y).
\]
More generally, each $F$ in $\Mod\sfT$ induces a functor
$F^*\colon\sfT^{\op}\to\Ab^*$ into the category of graded abelian
groups, with
\[
F^n(X)=F(\Si^{-n}X) \quad\text{for each $n\in\bbZ$}.
\]

\subsection*{Central ring actions}
Let $R$ be a graded commutative ring; thus $R$ is $\bbZ$-graded and
satisfies $rs=(-1)^{|r||s|}sr$ for each pair of homogeneous elements
$r,s$ in $R$.  We say that $\sfT$ is \emph{$R$-linear}, or that $R$
\emph{acts} on $\sfT$, if there is a homomorphism $\phi\col R\to
Z^*(\sfT)$ of graded rings, where
\[
Z^*(\sfT)=\bigoplus_{n\in\bbZ}\{\eta\col\Id_\sfT\to\Si^n\mid\eta\Si=(-1)^n\Si\eta\}
\]
is the graded centre of $\sfT$. For each object $X$ in $\sfT$ this
yields a homomorphism $\phi_X\col R\to\End^*_\sfT(X)$ of graded rings
such that for all objects $X,Y\in\sfT$ the $R$-module structures on
$\Hom^*_\sfT(X,Y)$ induced by $\phi_{X}$ and $\phi_{Y}$ agree up to a
sign. Namely, for any homogeneous elements $r\in R$ and
$\alpha\in\Hom^*_\sfT(X,Y)$, one has
\[
\phi_Y(r)\alpha=(-1)^{|r||\alpha|}\alpha\phi_X(r).
\]

Here are some examples.

\begin{example}
\label{ex:ring-actions}
(1) Any triangulated category admits a canonical action of $\bbZ$.

(2) The derived category of a ring $A$ has a canonical action of the centre of $A$.

(3) If $A$ is an algebra over a field $k$, the derived category of $A$
has a canonical action of the Hochschild cohomology of $A$ over $k$.

(4) Given a finite dimensional Hopf algebra $H$ over a field $k$ (for
example, the group algebra of a finite group), the derived category of
$H$ (and hence also the stable module category) has a canonical action
of the $k$-algebra $\Ext^*_{H}(k,k)$.
\end{example}

We fix an action of $R$ on $\sfT$.  The following observations will be used repeatedly.

\begin{remark}
  The $R$-action on $\sfT$ induces an action on any triangulated
  subcategory $\sfS\subseteq\sfT$ and on the quotient $\sfT/\sfS$,
  compatible with the inclusion and quotient functors,
  respectively. It also extends to an action on $\Mod\sfT$.
\end{remark}

\subsection*{Torsion objects}
The set of homogeneous prime ideals of $R$ is denoted $\Spec R$. For a
homogeneous ideal $\fa$ of $R$ we set
\[
\mcV(\fa)=\{\fp\in\Spec R\mid\fa\subseteq\fp\}.
\]   
Let $\mcV$ be a \emph{specialisation closed} subset of $\Spec R$; this
condition means that if $\mcV$ contains a prime $\fp$, then it
contains everything in $\mcV(\fp)$. An $R$-module $M$ is
\emph{$\mcV$-torsion} if $M_\fp=0$ for each $\fp\in\Spec
R\setminus\mcV$. Note that $M$ is $\mcV(\fa)$-torsion if and only if
each $r\in\fa$ and $x\in M$ satisfy $r^nx=0$ for $n\gg 0$.

A functor $F\in\Coh\sfT$ is \emph{$\mcV$-torsion} if $F^*(X)$ is
$\mcV$-torsion for all $X\in\sfT$. The full subcategory of all
$\mcV$-torsion functors is denoted by $(\Coh\sfT)_\mcV$. Analogously,
an object $Y\in\sfT$ is \emph{$\mcV$-torsion} if $\End_\sfT^*(Y)$ is
$\mcV$-torsion. This means the functor $\Hom_\sfT(-,Y)$ is
$\mcV$-torsion, since for each $X\in\sfT$ the $R$-action on
$\Hom^*_\sfT(X,Y)$ factors through $\End_\sfT^*(Y)$. Set
\[
\sfT_\mcV=\{X\in\sfT\mid \End_\sfT^*(X)_\fp=0\text{ for all
}\fp\in\Spec R\setminus\mcV\}.
\]
Note that $\sfT_\mcV$ is a thick subcategory of $\sfT$.  Recall that
we view $\Coh(\sfT_{\mcV})$ as a full subcategory of $\Coh \sfT$. It
follows from the definitions that there is an inclusion
\begin{equation}
\label{eq:coh-tor-is-tor-coh}
\Coh(\sfT_{\mcV})\subseteq (\Coh \sfT)_{\mcV}.
\end{equation}
Equality holds when $R$ is noetherian; see
Corollary~\ref{co:coh-tor-is-tor-coh}, and also
Propositions~\ref{pr:loc-prime} and \ref{pr:loc-closed}. Following the
definition in \eqref{eq:loc-fun}, the inclusion
$\sfT_\mcV\subseteq\sfT$ induces functors
\[
\gam_\mcV, L_\mcV\colon \Coh\sfT\lto\Coh\sfT,
\]
where $\gam_{\mcV}$ is a colocalisation functor and $L_{\mcV}$ is a
localisation functor. Note that these functors are exact and preserve
filtered colimits.

\subsection*{Inverting central elements}

Given a homogeneous element $r\in R$ of degree $d$ and $X\in\sfT$, we
write $\kos{X}{r}$ for the cone of the morphisms $X\xra{r}
\Si^{d}X$. This definition yields the following exact sequence.
\begin{equation}
\label{eq:koszul-object}
\cdots\lto H_X\xra{\ \pm r\ } H_{\Si^d X}\lto H_{\kos{X}{r}}
\lto H_{\Si X}\xra{\ \pm r\ } H_{\Si^{d+1} X}\lto \cdots
\end{equation}
In particular, inverting $r$ in $\sfT$ is equivalent to annihilating
$\kos Xr$ for all $X\in \sfT$.

Let $\Phi$ be a multiplicatively closed set of homogeneous elements in
$R$. The following lemma describes the quotient functor for $\sfT$
that inverts the elements of $\Phi$. We consider the specialisation
closed set
\[
\mcZ(\Phi)=\{\fp\in\Spec R\mid \fp\cap\Phi\neq\varnothing\}.
\]
Given an $R$-module $M$, we write $M[\Phi^{-1}]$ for the localisation
of $M$ with respect to $\Phi$. Note that $M$ is $\mcZ(\Phi)$-torsion
iff $M[\Phi^{-1}]=0$.

The following lemma is a variation of known results; see for instance
\cite[Theorem~3.6]{Balmer:2010a} or
\cite[Theorem~3.3.7]{Hovey/Palmieri/Strickland:1997a}.

\begin{lemma}
\label{le:loc-prime} 
Let $\Phi$ be a multiplicatively closed set of homogeneous elements in
$R$ and $\sfT'\subseteq\sfT$ a subcategory satisfying
$\Thick(\sfT')=\sfT$. Then there is an equality
\[
\sfT_{\mcZ(\Phi)}=\Thick(\{\kos{X}{r}\mid X\in\sfT',\,r\in\Phi\}),
\] 
and the quotient functor $\sfT\to\sfT/\sfT_{\mcZ(\Phi)}$ induces a natural isomorphism 
\[
\Hom^*_\sfT(X,Y)[\Phi^{-1}]\xra{\sim}\Hom^*_{\sfT/\sfT_{\mcZ(\Phi)}}(X,Y)
\]
for all objects $X,Y$ in $\sfT$.
\end{lemma}

\begin{proof}
Set $\sfS=\Thick(\{\kos{X}{r}\mid X\in\sfT',\,r\in\Phi\})$ and  $\sfU=\sfT/\sfS$. We claim: 
\begin{enumerate}
\item If $X$ or $Y$ is in $\sfS$, then $\Hom^*_\sfT(X,Y)[\Phi^{-1}]=0$.
\item For any $X,Y$ in $\sfT$, the natural morphism $\Hom^*_\sfT(X,Y)\to \Hom^*_\sfU(X,Y)$ induces an isomorphism
\[
\Hom^*_\sfT(X,Y)[\Phi^{-1}]\xra{\sim} \Hom^*_\sfU(X,Y).
\] 
\end{enumerate}
Indeed, (1) follows from \eqref{eq:koszul-object}. Given this, it
follows from the exact sequence \eqref{eq:loc-long} that the morphism
in (2) induces an isomorphism
\[
\Hom^*_\sfT(X,Y)[\Phi^{-1}]\xra{\sim} \Hom^*_\sfU(X,Y)[\Phi^{-1}].
\] 
On the other hand, $\Phi$ acts invertibly on $\Hom^*_\sfU(X,Y)$, that is to say,
\[
\Hom^*_\sfU(X,Y)\xra{\sim}\Hom^*_\sfU(X,Y)[\Phi^{-1}].
\] 
It suffices to check this claim for all $Y\in\sfT'$, and then it is
clear from \eqref{eq:koszul-object}. Combining both isomorphisms
yields (2), and completes the proof of the claims.

It remains to observe that $X\in\sfT$ is $\mcZ(\Phi)$-torsion iff
$\End_\sfT^*(X)[\Phi^{-1}]=0$; given (2) above, the latter condition
translates to $X=0$ in $\sfU$. Thus $\sfT_{\mcZ(\Phi)}=\sfS$.
\end{proof}

Let $\Phi$ be a multiplicatively closed set. We define a
functor \[L_\Phi\colon\Coh\sfT\lto\Coh\sfT\] by taking $F$ in
$\Coh\sfT$ to $F[\Phi^{-1}]$ given by $F[\Phi^{-1}]^*(X)=
F^*(X)[\Phi^{-1}]$ for $X\in\sfT$.  It is easy to verify that this is
an exact localisation functor; the corresponding colocalisation
functor is denoted $\gam_\Phi$.

\begin{proposition}
\label{pr:loc-prime}
  There is a natural isomorphism $L_\Phi\xra{\sim} L_{\mcZ(\Phi)}$ and hence
\[
\Coh(\sfT_{\mcZ(\Phi)})=(\Coh \sfT)_{\mcZ(\Phi)}.
\]
\end{proposition}

\begin{proof}
  Lemma~\ref{le:loc-prime} yields the isomorphism for
  representable functors, and the general case follows since
  $L_\Phi$ and $L_{\mcZ(\Phi)}$ preserve filtered colimits.
\end{proof}

\subsection*{Localisation at a prime ideal}

Let $\fp$ be a homogeneous prime ideal of $R$. Thus $R\setminus\fp$ is
a multiplicatively closed subset and
\[\mcZ(R\setminus\fp)=\{\fq\in\Spec R\mid \fq\not\subseteq\fp\}.\]
Set \[\sfT_{\fp} = \sfT/\sfT_{\mcZ(R\setminus\fp)}\] and let $X_{\fp}$
denote the image of an object $X$ in $\sfT$ under the natural functor
$\sfT\to\sfT_{\fp}$. This quotient category is described in
Lemma~\ref{le:loc-prime}.  Thus for all $X,Y\in\sfT$ there is a
natural isomorphism
\[
\Hom^*_\sfT(X,Y)_\fp\cong \Hom^*_{\sfT_\fp}(X_\fp,Y_\fp).
\]
It follows from Proposition~\ref{pr:loc-prime} that the localisation functor 
\[
\Coh\sfT\lto\Coh\sfT,\quad F\mapsto F_\fp,
\]
defined by $(F_\fp)^*(X)=F^*(X)_\fp$ is isomorphic to
$L_{\mcZ(R\setminus\fp)}$. A functor $F\in\Coh\sfT$ is called
\emph{$\fp$-local} if $F\cong F_\fp$, and $(\Coh\sfT)_\fp$ denotes the
full subcategory formed by all $\fp$-local
functors. Proposition~\ref{pr:loc-prime} yields an identification
\begin{equation}
\label{eq:loc-prime}
\Coh(\sfT_\fp)=(\Coh\sfT)_\fp.
\end{equation}

\subsection*{Koszul objects}

Fix a homogeneous element $r\in R$ of degree $d$.  For each $X$ in
$\sfT$ and each integer $n$ set $X_n=\Si^{nd}X$ and consider the
commuting diagram
\[
\xymatrix{
X\ar@{=}[r]\ar[d]^{r^{}} &X\ar@{=}[r]\ar[d]^{r^{2}} &X\ar@{=}[r]\ar[d]^{r^{3}} &\cdots\\
X_1\ar[d]\ar[r]^r        &X_2\ar[d]\ar[r]^r         &X_3\ar[d]\ar[r]^r         &\cdots\\ \kos
Xr\ar[r]                 &\kos Xr^{2}\ar[r]         &\kos Xr^{3}\ar[r]         &\cdots }
\] 
where each vertical sequence is given by the exact triangle defining
$\kos Xr^{n}$, and the morphisms in the last row are the
(non-canonical) ones induced by the commutativity of the upper
squares.

\begin{lemma}\label{le:gamma_r_colim}
  Let $r\in R$ be a homogeneous element of degree $d$. 
\begin{enumerate}
\item For $F\in\Coh\sfT$, the colimit of the sequence 
\[
F \stackrel{r}\lto\Si^{d} F \stackrel{r}\lto\Si^{2d} F \stackrel{r}\lto\Si^{3d} F \stackrel{r}\lto\cdots
\] 
is naturally isomorphic to $L_{\mcV(r)} F$.
\item For $X\in\sfT$, the colimit of the sequence
  \[
  H_{\kos{\Si^{-1}X}r^1}\lto H_{\kos{\Si^{-1}X}r^2}\lto H_{\kos{\Si^{-1}X}r^3}\lto\cdots
  \] 
  is naturally isomorphic to $\gam_{\mcV(r)}H_X$.
\end{enumerate}
\end{lemma}

\begin{proof}
  The colimit construction in (1) yields a functor
  $\Coh\sfT\to\Coh\sfT$; we claim that it is isomorphic to
  $L_{\mcV(r)}$.  It suffices to prove this for representable functors
  as both functors preserve filtered colimits. When $F=H_X$ one has an
  exact sequence
  \[
  \cdots\lto\colim H_{\kos{\Si^{-1}X}r^n}\lto H_X \lto \colim H_{X_n}\lto \colim H_{\kos Xr^n}\lto\cdots
  \] 
  where $r$ acts invertibly on $\colim H_{X_n}$ while $\colim H_{\kos
    Xr^n}$ is $\mcV(r)$-torsion. Thus the sequence is isomorphic to
  the localisation sequence for $\sfT_{\mcV(r)}\subseteq\sfT$, by
  Propositions~\ref{pr:loc-seq} and \ref{pr:loc-prime}.
\end{proof}

Let $\fa$ be a finitely generated homogeneous ideal of $R$ and
$X\in\sfT$. Pick a sequence of elements $r_1,\ldots,r_n$ in $R$ that
generate the ideal $\fa$ and define inductively
\[
X_0=X\qquad\text{and}\qquad X_{i}=\kos {X_{i-1}}{r_i} \text{ for }1\le i\le n.
\] 
We call $X_n$ a \emph{Koszul object} of $X$ with respect to $\fa$, and
denote it $\kos X{\fa}$. This depends on a choice of a sequence of
generators for $\fa$, so our notation is ambiguous. However, there is
the following uniqueness result.

\begin{lemma}
\label{le:kos}
There is an equality
\[
\Thick(\kos X\fa)=\{Y\in\Thick(X)\mid \End_\sfT^*(Y)_\fp=0\text{ for
  all } \fp\not\supseteq\fa\}.
\]
\end{lemma}

\begin{proof}
  When $\fa$ is generated by a single element the desired statement
  follows from Lemma~\ref{le:loc-prime}, applied to $\sfT=\Thick(X)$.
  An iteration settles the general case.
\end{proof}

\begin{proposition}
\label{pr:loc-closed}
Let $\fa$ be a finitely generated homogeneous ideal of $R$. Then
\[
\sfT_{\mcV(\fa)}=\Thick(\{\kos X\fa\mid
X\in\sfT\})\qquad\text{and}\qquad
\Coh(\sfT_{\mcV(\fa)})=(\Coh\sfT)_{\mcV(\fa)}.
\] 
Moreover, the objects of $\sfT_{\mcV(\fa)}$ are precisely the direct
summands of Koszul objects $\kos{X}{\fb}$ with $X\in\sfT$ and $\fb$ an
ideal of $R$ satisfying $\sqrt\fb=\sqrt \fa$.
\end{proposition}

\begin{proof}
 Set $\sfS=\Thick(\{\kos X\fa\mid X\in\sfT\})$. It suffices to show that 
 \[
 \Coh(\sfT_{\mcV(\fa)})\subseteq (\Coh\sfT)_{\mcV(\fa)}\subseteq
 \Coh\sfS \subseteq\Coh(\sfT_{\mcV(\fa)}).
 \] 
 From this the first part of the assertion follows. The fact that all
 objects in $\sfT_{\mcV(\fa)}$ are direct summands of Koszul objects
 follows from the proof.
 
 The first inclusion is by definition and the last follows from
 Lemma~\ref{le:kos}. To verify the inclusion in the middle, it
 suffices to show that for any $F\in (\Coh\sfT)_{\mcV(\fa)}$ each
 morphism $\phi\colon H_X\to F$ with $X$ in $\sfT$ factors through a
 morphism $H_X\to H_Y$ with $Y$ in $\sfS$; see
 Lemma~\ref{le:cofinal}. To this end, let $r_{1},\dots,r_{n}$ be a
 sequence of elements that generate the ideal $\fa$. Starting with
 $X_{0}=X$ and $\phi_{0}=\phi$, we construct factorisations
\[
\phi_{i-1}\colon H_{X_{i-1}}\lto H_{X_{i}}\xra{\phi_{i}} F
\] 
for $i=1,2,\ldots, n$.  The assumption on $F$ implies that each
$\phi_{i-1}$ is annihilated by $r_i^{\alpha_i}$, for some
$\alpha_{i}\ge 1$. Thus we set
\[
X_{i}=\Si^{-\alpha_i|r_i|}\kos {X_{i-1}}{r_i^{\alpha_i}}.
\] 
The object $Y=X_{n}$ is the desired object; it belongs to $\sfS$ by
Lemma~\ref{le:kos}.
\end{proof}

\subsection*{Composition laws}

We show that cohomological localisation and colocalisation functors
commute; see Lemma~\ref{le:rules-commute} for related commutation rules.

\begin{lemma}\label{le:commute}
  Let $\Phi$ and $\Psi$ be multiplicatively closed sets of homogeneous
  elements in $R$. Then \[ L_\Phi\comp L_\Psi\cong L_\Psi\comp L_\Phi,
  \qquad L_\Phi\comp \gam_\Psi\cong \gam_\Psi\comp L_\Phi,
  \qquad \gam_\Phi\comp \gam_\Psi\cong \gam_\Psi\comp \gam_\Phi.\]
\end{lemma}
\begin{proof}
  The first isomorphism is clear since localising $R$-modules with
  respect to $\Phi$ and $\Psi$ commutes.  We consider the exact
  localisation sequence
\[\xymatrix{
\mathbb E_\Psi\colon&\cdots\ar[r]&\gam_\Psi\ar[r]&\Id\ar[r]&L_\Psi\ar[r]&\cdots
}\]
and the pair of morphisms \[(L_\Phi \mathbb E_\Psi)\lto (L_\Phi
\mathbb E_\Psi) L_\Phi=L_\Phi (\mathbb E_\Psi L_\Phi)\longleftarrow
(\mathbb E_\Psi L_\Phi).\] This yields the following commutative diagram
with exact rows.
\[\xymatrix{
\cdots\ar[r]&L_\Phi\gam_\Psi\ar[r]\ar[d]&L_\Phi\ar[r]\ar[d]&L_\Phi L_\Psi\ar[r]\ar[d]&\cdots\\
\cdots\ar[r]&L_\Phi\gam_\Psi L_\Phi\ar[r]&L_\Phi L_\Phi\ar[r]&L_\Phi L_\Psi L_\Phi\ar[r]&\cdots\\
\cdots\ar[r]&\gam_\Psi L_\Phi\ar[r]\ar[u]&L_\Phi\ar[r]\ar[u]&L_\Psi L_\Phi\ar[r]\ar[u]&\cdots
}\]
In two of three columns the vertical morphisms are isomorphisms. Thus
the five lemma implies that $L_\Phi\gam_\Psi\cong \gam_\Psi L_\Phi$. A
similar argument based on the pair of morphisms $\mathbb
E_\Psi\gam_\Phi\leftarrow \gam_\Phi \mathbb E_\Psi \gam_\Phi\to
\gam_\Phi\mathbb E_\Psi$ is used to deduce the third isomorphism
$\gam_\Phi\gam_\Psi\cong \gam_\Psi \gam_\Phi$ from the second.
\end{proof}

Next observe for a multiplicatively closed set $\Phi=\{r^i\mid
i\in\bbN \}$ that $\mcZ(\Phi)=\mcV(r)$ and so
$\gam_{\Phi}\cong\gam_{\mcV(r)}$ by Proposition~\ref{pr:loc-prime}.

\begin{lemma}\label{le:composite}
  Let $\fa$ and $\fb$ be finitely generated homogeneous ideals of $R$. Then 
  \[
 \gam_{\mcV(\fa)}\comp \gam_{\mcV(\fb)}\cong \gam_{\mcV(\fa)\cap\mcV(\fb)}.
 \]
\end{lemma}

\begin{proof}
It suffices to show for $\fa$ generated by
  homogeneous elements $r_1,\ldots, r_n$ that
\[
\gam_{\mcV(\fa)}=\gam_{\mcV(r_n)}\comp\ldots\comp \gam_{\mcV(r_1)}.
\] We prove this assertion by induction on $n$.  Let
$\fa'=(r_1,\ldots, r_{n-1})$.  We know from Lemma~\ref{le:commute}
that the $\gam_{\mcV(r_i)}$ commute. Thus $\gam_{\mcV(\fa')}
\gam_{\mcV(r_n)}\cong \gam_{\mcV(r_n)} \gam_{\mcV(\fa')}$ by the
induction hypothesis. Using Proposition~\ref{pr:loc-closed}, it
follows that the image of $\gam_{\mcV(r_n)} \gam_{\mcV(\fa')}$ belongs
to
\begin{align*}
\Coh(\sfT_{\mcV(r_n)})\cap\Coh(\sfT_{\mcV(\fa')})&=(\Coh\sfT)_{\mcV(r_n)}\cap (\Coh\sfT)_{\mcV(\fa')}\\
&=(\Coh\sfT)_{\mcV(\fa)}\\
&=\Coh(\sfT_{\mcV(\fa)}).
\end{align*}
Therefore $L_{\mcV(\fa)}(\gam_{\mcV(r_n)} \gam_{\mcV(\fa')})=0$. On
the other hand, $\gam_{\mcV(\fa)}(\gam_{\mcV(r_n)}
\gam_{\mcV(\fa')})=\gam_{\mcV(\fa)}$.  The exact localisation sequence
\eqref{eq:loc-seq} yields the following exact sequence
\[
\cdots\to \gam_{\mcV(\fa)}(\gam_{\mcV(r_n)}
\gam_{\mcV(\fa')}) \to \gam_{\mcV(r_n)} \gam_{\mcV(\fa')}\to
L_{\mcV(\fa)}(\gam_{\mcV(r_n)}
\gam_{\mcV(\fa')})\to\cdots
\]
and therefore $\gam_{\mcV(\fa)}\cong\gam_{\mcV(r_n)}
\gam_{\mcV(\fa')}$.
\end{proof}

\begin{corollary}
\label{co:commute}
Let $\fa$ be a finitely generated homogeneous ideal, and $\fp$ a
homogeneous prime ideal, of $R$.  For each $F\in\Coh\sfT$, there is a
natural isomorphism
\[
  (\gam_{\mcV(\fa)} F)_\fp\cong \gam_{\mcV(\fa)}( F_\fp).
\]
\end{corollary}

\begin{proof}
Apply Lemmas~\ref{le:commute} and \ref{le:composite}.
\end{proof}

\section{Support}
In this section, we define the support of a cohomological functor and
establish some useful rules for computing it; the development
parallels the one in \cite[\S5]{\bik:2008a}.

Let $R$ be a graded commutative ring and $\sfT$ be an essentially
small $R$-linear triangulated category. From now on we assume $R$ to
be noetherian.

\subsection*{Support}

For each $F$ in $\Coh\sfT$ and $\fp$ in $\Spec R$ set 
\[
\gam_\fp F=\gam_{\mcV(\fp)}( F_\fp).
\] 
Then $\gam_{\fp}$ is an exact functor on $\Coh\sfT$ that preserves
filtered colimits. The subset
\[
\supp_R F=\{\fp\in\Spec R\mid \gam_\fp F\neq 0\}
\]
is called the \emph{support} of $F$.

\begin{proposition}
\label{pr:support}
Let $F\in\Coh\sfT$. Then $\supp_R F=\varnothing$ if and only if $F=0$.
\end{proposition}

\begin{proof}
  Clearly, $\supp_R F=\varnothing$ when $F=0$.  Suppose $F$ is
  non-zero. Recall that if an $R$-module $M$ is non-zero, then there
  exists a $\fp\in \Spec R$ such that $M_{\fp}\ne 0$. Choose a prime
  $\fp$ that is minimal subject to the condition that $F_\fp\neq
  0$. Then for all primes $\fq$ properly contained in $\fp$ and all
  $X\in\sfT$, one has
 \[
 F_{\fp}(X)_{\fq}\cong F(X)_{\fq}=0.
 \]
 Hence $F_{\fp}(X)$ is $\mcV(\fp)$-torsion, by
 \cite[Theorem~6.5]{Matsumura:1986a}; that is to say, $F_\fp$ is
 $\mcV(\fp)$-torsion. It then follows from
 Proposition~\ref{pr:loc-closed} that $F_{\fp}$ is in
 $\Coh(\sfT_{\mcV(\fp)})$, so the natural map $\gam_\fp F\to F_\fp$ is
 an isomorphism. As $F_{\fp}\neq 0$, one gets that $\fp$ is in
 $\supp_{R}F$.
\end{proof}

We can compute the support of the representable functors as follows.

\begin{proposition}
\label{pr:tria_support}
Let $X$ be an object in $\sfT$. Then
\[
\supp_R H_X\subseteq\{\fp\in\Spec R\mid \End^*_\sfT(X)_\fp\neq 0\},
\]
and equality holds when $\End^*_\sfT(X)$ is finitely generated over
$R$.
\end{proposition}

\begin{proof}
  The inclusion holds because $(H_X)_\fp=0$ iff $\End^*_\sfT(X)_\fp=
  0$. Now suppose that $\End^*_\sfT(X)_\fp\neq 0$ and that
  $\End^*_\sfT(X)$ is finitely generated. Then $X_\fp\neq 0$, and an
  application of Nakayama's lemma gives $\kos{X_\fp}\fp\neq 0$; see
  \cite[Lemma~5.11]{\bik:2008a}. The functor $H_{\kos{X_\fp}\fp}$ is
  $\mcV(\fp)$-torsion by Lemma~\ref{le:kos}. It is also $\fp$-local,
  so one gets
 \[
 \gam_\fp H_{\kos{X}\fp}=\gam_{\mcV(\fp)}H_{\kos{X_\fp}\fp}= H_{\kos{X_\fp}\fp}\neq 0.
 \] 
Hence $\fp$ is in $\supp_{R} H_X$.
\end{proof}

\subsection*{Composition laws}
Computing  support is compatible with cohomological localisation and colocalisation.

\begin{proposition}
\label{pr:support-V}
Let $\mcV\subseteq\Spec R$ be a specialisation closed subset. For each
$F\in\Coh\sfT$ the following equalities hold
\begin{align*}
\supp_R\gam_\mcV F &= \mcV\cap \supp_R F,\\
\supp_R L_\mcV F&=(\Spec R\setminus \mcV)\cap\supp_R F.
\end{align*}
\end{proposition}

\begin{proof}
  If $\fp\not\in\mcV$ then $(\gam_\mcV F)_\fp=0$,  since
  $\gam_{\mcV}F$ is $\mcV$-torsion. Thus $\supp_R\gam_\mcV
  F \subseteq \mcV$.  If $\fp\in \mcV$ then
  $\sfT_{\mcV(\fp)}\subseteq\sfT_{\mcV}$, hence
  $\gam_{\mcV(\fp)}\gam_{\mcV}=\gam_{\mcV(\fp)}$; see
  Lemma~\ref{le:rules-commute}. This gives the second equality below:
\[
\gam_\fp (\gam_\mcV F)=(\gam_{\mcV(\fp)}\gam_\mcV F)_\fp=
(\gam_{\mcV(\fp)}F)_\fp = \gam_\fp F,
\] 
while the other two are by Corollary~\ref{co:commute}.  Thus 
\[
\supp_R\gam_\mcV F= \mcV\cap\supp_R\gam_\mcV F =\mcV\cap \supp_R F.
\]
This proves the first equality; the proof of the second is similar.  
\end{proof}

The following result says that a $\mcV$-torsion functor is a colimit
of representable functors defined by $\mcV$-torsion objects.
 
\begin{corollary}
\label{co:coh-tor-is-tor-coh}
Let $\mcV\subseteq\Spec R$ be specialisation closed. Then
$\Coh(\sfT_{\mcV}) = \Coh(\sfT)_{\mcV}$.
\end{corollary}

\begin{proof}
  It suffices to prove that $\Coh(\sfT)_{\mcV}\subseteq
  \Coh(\sfT_{\mcV})$; confer \eqref{eq:coh-tor-is-tor-coh}. Fix
  $F\in(\Coh \sfT)_{\mcV}$. For any $\fp\not\in\mcV$, one has
  $F_{\fp}=0$, and hence $\gam_{\fp}F=0$, that is to say,
  $\supp_{R}F\subseteq \mcV$. Proposition~\ref{pr:support-V} then
  implies that $L_{\mcV}F=0$ and it follows from \eqref{eq:loc-seq}
  that the natural map $\gam_{\mcV}F\to F$ is an isomorphism. This is
  the desired result.
\end{proof}

\begin{corollary}
\label{co:gamma-L}
Let $\mcV$ and $\mcW$ be specialisation closed subsets of $\Spec R$. Then
\[
\gam_\mcV\gam_\mcW\cong \gam_{\mcV\cap\mcW}\cong
\gam_\mcW\gam_\mcV,\qquad L_\mcV L_\mcW\cong L _{\mcV\cup\mcW}\cong
L_\mcW L_\mcV,\qquad \gam_\mcV L_\mcW\cong L_\mcW\gam_\mcV.
\]
\end{corollary}

\begin{proof}
  Given Proposition~\ref{pr:support-V}, one can argue as for
  \cite[Proposition~6.1]{\bik:2008a}.
\end{proof}

\begin{corollary}
\label{co:gammap-variants}
Let $\fp\in\Spec R$, and let $\mcV$ and $\mcW$ be specialisation
closed subsets of $\Spec R$ such that
$\mcV\setminus\mcW=\{\fp\}$. Then
\[
 L_{\mcW}\gam_{\mcV}\cong \gam_\fp\cong \gam_{\mcV}L_{\mcW}.
 \]
\end{corollary} 

\begin{proof}
  Given Proposition~\ref{pr:support-V}, one can argue as in the proof
  of \cite[Theorem~6.2]{\bik:2008a}.
\end{proof}

\begin{corollary}
\label{co:gammap-support}
Let $\fp\in\Spec R$ and $F\in\Coh\sfT$. Then  $\supp_{R}\gam_\fp F\subseteq \{\fp\}$.
\end{corollary} 

\begin{proof}
  Combine Proposition~\ref{pr:support-V} and
  Corollary~\ref{co:gammap-variants}.
\end{proof}

\subsection*{Colimits}
The following result can be used to reduce computations involving
specialisation closed subsets to those involving closed sets; it is an
analogue of \cite[Lemma~6.6]{Stevenson:2011a} in the compactly
generated context.

\begin{lemma}
\label{le:gamma_colim}
Let $\mcV=\bigcup_{\alpha}\mcV_\alpha$ be a directed union of
specialisation closed subsets of $\Spec R$. Then
 \[
 \colim\gam_{\mcV_\alpha}\xra{\sim}\gam_{\mcV} \qquad\text{and}\qquad
 \colim L_{\mcV_\alpha}\xra{\sim} L_{\mcV}.
 \]
\end{lemma}

\begin{proof}
  We make repeated use of Proposition~\ref{pr:support}. Since
  $\gam_{\mcV_\alpha}\gam_\mcV=\gam_{\mcV_\alpha}$, it follows from
  the localisation sequence \eqref{eq:loc-seq} that the natural
  morphism $\colim\gam_{\mcV_\alpha}\to\gam_{\mcV}$ fits into an exact
  sequence
  \[
 \cdots \lto\colim\gam_{\mcV_\alpha}\lto\gam_{\mcV}\lto\colim
 L_{\mcV_\alpha}\gam_\mcV\lto\cdots
 \]
 We claim that $\colim L_{\mcV_\alpha}\gam_\mcV=0$. Indeed,
 $\gam_{\fp}$ commutes with filtered colimits so the claim is that
 $\colim (\gam_{\fp} L_{\mcV_\alpha}\gam_\mcV) =0$ for each $\fp$ in
 $\Spec R$. Proposition~\ref{pr:support-V} and its corollaries yield
\[
\supp_{R} (\gam_{\fp} L_{\mcV_\alpha}\gam_\mcV F)\subseteq \{\fp\}\cap
(\Spec R\setminus \mcV_{\alpha})\cap \mcV
\]
for each $F$ in $\Coh\sfT$. Thus, if $\fp\not\in\mcV$, then evidently
$\gam_{\fp} L_{\mcV_\alpha}\gam_\mcV = 0$. Assume $\fp\in\mcV$.  When
$\fp\in\mcV_\alpha$ as well, it again follows from the equality above
that $\gam_{\fp} L_{\mcV_\alpha}\gam_\mcV = 0$. Since $\mcV$ is
directed union of the $\mcV_{\alpha}$, the desired vanishing follows.

This proves the claim. The exact sequence above then yields the
isomorphism involving $\gam_{\mcV}$. The assertion for $L_\mcV$
follows, using again \eqref{eq:loc-seq}.
\end{proof}

\section{The local-global principle}

Let $R$ be a noetherian graded commutative ring and $\sfT$ be an
essentially small $R$-linear triangulated category.  In this section
we establish a local-global principle for $\Coh\sfT$, analogous to the
one in \cite[\S3]{\bik:2011a}. A local-global principle for $\sfT$
then follows.

\subsection*{Localising subcategories}
We call a full subcategory of $\Coh\sfT$ \emph{localising} if it is
closed under forming coproducts, extensions, and suspensions. Here,
$F$ is an \emph{extension} of $F'$ and $F''$ if there is an exact
sequence $F'\to F\to F''$ in $\Coh\sfT$. Any localising subcategory is
closed under subobjects and quotient objects; this follows by
specialising $F'=0$ or $F''=0$. In particular, a localising
subcategory is closed under filtered colimits. The smallest localising
subcategory containing a subcategory $\sfC$ of $\sfT$ is denoted
$\Loc(\sfC)$.

The following lemma provides some basic properties of localising
subcategories; they will be used without further mention.  The
argument is straightforward.

\begin{lemma}
Let $P\colon\Coh\sfT\to\Coh\sfU$ be an exact functor that preserves coproducts.
\begin{enumerate}
\item If $\sfC\subseteq\Coh\sfU$ is localising, then
  $\{F\in\Coh\sfT\mid P(F)\in\sfC\}$ is a localising subcategory of
  $\Coh\sfT$. In particular, the kernel \[\Ker(P)=\{F\in\Coh\sfT\mid
  P(F)=0\}\]
  is a localising subcategory of $\Coh \sfT$.
\item Any subcategory $\sfC\subseteq\Coh\sfT$ satisfies
  $P(\Loc(\sfC))\subseteq \Loc(P(\sfC))$. \qed
\end{enumerate}
\end{lemma}

We provide important examples of localising subcategories.

\begin{lemma}
\label{le:localising}
Let $\sfS\subseteq\sfT$ be a triangulated subcategory. Then $\Coh\sfS$
and $\Coh\sfT/\sfS$ are localising when viewed as subcategories of
$\Coh\sfT$.
\end{lemma}

\begin{proof}
  From Corollary~\ref{co:loc-seq} it follows that $\Coh\sfS$ equals
  the kernel of the functor $L$, while $\Coh\sfT/\sfS$ equals the
  kernel of the functor $\gam$. It remains to observe that both
  functors are exact and preserve coproducts, by
  Lemma~\ref{le:exact-fun}.
\end{proof}

\begin{proposition}
\label{pr:colim}
Let $\mcV$ be a specialisation closed subset of $\Spec R$ and
$F\in\Coh\sfT$. Then $\gam_\mcV F$ and $L_\mcV F$ belong to $\Loc(F)$.
\end{proposition}

\begin{proof}
  From the localisation sequence \eqref{eq:loc-seq} it follows that it
  suffices to prove this for $\gam_\mcV$. The assertion follows from
  Lemmas~\ref{le:gamma_r_colim} and \ref{le:composite} when
  $\mcV=\mcV(\fa)$ for some ideal $\fa$.  The general case then
  follows from Lemma~\ref{le:gamma_colim}.
\end{proof}

\begin{lemma}
\label{le:Loc}
Let $X\in\sfT$ and $\fa$ be a homogeneous ideal of $R$. Then
\[
\Loc(H_X)=\Coh\Thick(X)\qquad\text{and}\qquad\Loc(\gam_{\mcV(\fa)}H_X)=\Loc(H_{\kos X\fa}).
\]
\end{lemma}

\begin{proof}
Both assertions use the following observation:
\[
Y\in\Thick(X)\quad \implies\quad H_{Y}\in \Loc(H_{X}).
\]
From this it follows that $\Coh\Thick(X)\subseteq \Loc(H_X)$, while
the reverse inclusion is by Lemma~\ref{le:localising}. This settles
the first of the desired equalities.

For the second one, note that $\gam_{\mcV(r)} H_X\in \Loc(H_{\kos X
  r})$ for any homogeneous element $r\in R$, by
Lemma~\ref{le:gamma_r_colim}. Thus an induction on the number of
generators of $\fa$ shows that $\gam_{\mcV(\fa)}H_X\in\Loc(H_{\kos
  X\fa})$. Conversely, $H_{\kos X \fa}$ belongs to $\Loc(H_X)$, by the
observation above, so $\gam_{\mcV(\fa)}H_{\kos X \fa}\cong H_{\kos X
  \fa}$ belongs to $\Loc(\gam_{\mcV(\fa)} H_X)$.
\end{proof}

\subsection*{The local-global principle}

The following result is the analogue of a local-global principle for
compactly generated triangulated categories
\cite[Theorem~3.1]{\bik:2011a}.

\begin{theorem}[Local-global principle]
\label{th:lgp_coh}
Let $F\in\Coh\sfT$. Then
\[
\Loc(F) =\Loc(\{\gam_\fp F\mid\fp\in\Spec R\})=\Loc(\{F_\fp\mid\fp\in\Spec R\}).
\]
\end{theorem}

\begin{proof}
  The idea for this proof is taken from
  \cite[Lemma~2.10]{Neeman:1992b}.  

It follows from
  Proposition~\ref{pr:colim} that $\gam_\fp F$ and $F_\fp$ belong to
  $\Loc(F)$.

Now we set $\sfC=\Loc(\{\gam_\fp F\mid\fp\in\Spec R\})$ and prove that
$F\in\sfC$. Consider the set of specialisation closed subsets $\mcW$
of $\Spec R$ such that $\gam_\mcW F\in\sfC$. This set is non-empty,
for it contains the empty set, and it is closed under directed unions
by Lemma~\ref{le:gamma_colim}.  Thus it has a maximal element, say
$\mcV$, by Zorn's lemma. We claim that $\mcV=\Spec R$. To this end
assume $\mcV\neq\Spec R$ and choose a prime ideal $\fp$ maximal in the
subset $\Spec R\setminus \mcV$. The subset $\mcV\cup\{\fp\}$ is then
specialisation closed. Consider the localisation sequence with respect
to $\mcV$:
\[
\cdots\lto\gam_\mcV\gam_{\mcV\cup\{\fp\}}F\lto
\gam_{\mcV\cup\{\fp\}}F\lto L_\mcV\gam_{\mcV\cup\{\fp\}}F\lto\cdots
\] 
Corollaries~\ref{co:gamma-L} and \ref{co:gammap-variants} yield that
$\gam_\mcV\cong\gam_\mcV\gam_{\mcV\cup\{\fp\}}$ and $\gam_\fp\cong
L_\mcV\gam_{\mcV\cup\{\fp\}}$. Hence the terms on the left and on the
  right of the localisation sequence are in $\sfC$, and hence so is
  the one in the middle, $\gam_{\mcV\cup\{\fp\}}F$. This contradicts
  the maximality of $\mcV$ and yields the first equality.

The second equality follows from the first since $\gam_\fp
F=\gam_{\mcV(\fp)}(F_\fp)\in\Loc(F_\fp)$ for each prime $\fp$, by
Proposition~\ref{pr:colim}.
\end{proof}

Recall that the $R$-action on $\sfT$ induces an action on any
triangulated subcategory $\sfS\subseteq \sfT$ and on $\sfT/\sfS$. It
is not hard to see that the induced functor $\sfS_{\fp}\to\sfT_{\fp}$
is exact, full, and faithful. For this reason, we view $\sfS_{\fp}$ as
a triangulated subcategory of $\sfT_{\fp}$.

\begin{lemma}
\label{le:local-prime}
Let $\sfS\subseteq \sfT$ be a triangulated subcategory and
$\fp\in\Spec R$.  Then the canonical functor $\sfT_\fp\to
(\sfT/\sfS)_\fp$ induces an equivalence of triangulated categories
\[
\sfT_\fp/\sfS_\fp\stackrel{\sim}\lto (\sfT/\sfS)_\fp.
\]
\end{lemma}

\begin{proof}
  Consider the commutative diagrams below. The one on the left
  is clear from the constructions and induces the one on the right.
\begin{gather*}
\begin{gathered}
\xymatrix{
\sfT\ar[r]^-f\ar[d]^p&\sfT/\sfS\ar[d]^q\\
\sfT_\fp\ar[r]^-{f_\fp}&(\sfT/\sfS)_\fp}
\end{gathered}
\qquad\qquad
\begin{gathered}
\xymatrix{
\Coh \sfT \ar@{<-}[r]^-{f_{*}}\ar@{<-}[d]^{p_{*}} & \Coh \sfT/\sfS\ar@{<-}[d]^-{q_{*}}\\
\Coh(\sfT_\fp)\ar@{<-}[r]^-{(f_\fp)_{*}}&\Coh(\sfT/\sfS)_\fp}
\end{gathered}
\end{gather*}
Since $f$ and $q$ are quotient functors, so is their composition
$qf$. Then $f_{\fp}p$ and $p$ are quotient functors, and from this it
is not hard to verify that so is $f_{\fp}$. We claim that its kernel
consists precisely of direct summands of objects of $\sfS_{\fp}$, and
the statement would then follow.

As to the claim: In the display above, the functors in the square on
the right are all fully faithful, since they are induced by quotient
functors. Hence we view all the categories in the diagram as
subcategories of $\Coh \sfT$. Recall from Corollary~\ref{co:loc-seq}
that $\Ker(f^{*}) = \Coh \sfS$. Using \eqref{eq:loc-prime} it follows
that $\Ker(f_{\fp}^{*}) = \Coh (\sfS_{\fp})$.  In particular, for an
object $X\in\sfT_{\fp}$, one has $f_\fp(X)=0$ iff $f_{\fp}^{*}(H_X)=0$
iff $X$ is a direct summand of an object in $\sfS_\fp$; the second
equivalence is by Lemma \ref{le:cofinal}. This justifies the claim.
\end{proof}

For $\fp\in\Spec R$ we set $\gam_\fp\sfT=(\sfT_\fp)_{\mcV(\fp)}$. This
yields the following diagram
\[
\xymatrix{
\sfT  \ar@{->>}[r]&\sfT_\fp & \ar@{>->}[l]\,\gam_\fp\sfT
}\]
and henceforth we make the identification 
\[
\Coh\gam_\fp\sfT=(\Coh\sfT)_\fp\cap (\Coh\sfT)_{\mcV(\fp)}.
\]

\begin{corollary} 
\label{co:classification}
Taking a localising subcategory $\sfC\subseteq\Coh\sfT$ to the family 
\[
(\sfC\cap\Coh\gam_\fp\sfT)_{\fp\in\Spec R}
\]
induces a bijection between
\begin{enumerate}
\item[--] the localising subcategories of $\Coh\sfT$, and
\item[--] the families $(\sfC(\fp))_{\fp\in\Spec R}$ with $\sfC(\fp)$
a localising subcategory of $\Coh\gam_{\fp}\sfT$.
\end{enumerate}
\end{corollary}

\begin{proof}
  The inverse map takes a family $(\sfC(\fp))_{\fp\in\Spec R}$ to the
  smallest localising subcategory of $\Coh\sfT$ containing all
  $\sfC(\fp)$.
\end{proof}

\begin{remark}\label{re:classification}
  There is an analogue of Corollary~\ref{co:classification} for thick
  subcategories of $\sfT$ since each thick subcategory
  $\sfS\subseteq\sfT$ is determined by the localising subcategory
  $\Coh\sfS\subseteq\Coh\sfT$; see Lemma~\ref{le:localising}.
\end{remark}

\subsection*{Consequences of the local-global principle}
For $X\in\sfT$ and $\fp\in\Spec R$, we set $X(\fp)=(\kos X \fp)_\fp$
and identify this with $\kos{(X_\fp)} \fp$.  Note that
Lemma~\ref{le:Loc} implies
\begin{equation}
\label{eq:kos} 
\Loc(\gam_\fp H_X)=\Loc(H_{X(\fp)}).
\end{equation}

The following is the local-global principle for $\sfT$ announced in
the introduction.

\begin{theorem}[Local-global principle]
\label{th:lgp}
Let $\sfS$ be a thick subcategory of $\sfT$. Then the following
conditions are equivalent for an object $X$ in $\sfT$:
\begin{enumerate}
\item $X$ belongs to $\sfS$.
\item $X_\fp$ belongs to $\Thick(\sfS_\fp)$ for each $\fp\in\Spec R$.
\item $X(\fp)$ belongs to $\Thick(\sfS_\fp)$ for each $\fp\in\Spec R$.
\end{enumerate}
\end{theorem}

\begin{proof}
Evidently (1) $\Ra$ (2) and (2) $\Ra$ (3).

Assume (3) holds. We work in $\Coh\sfT$. For each $\fp\in\Spec R$, the
hypothesis implies the first inclusion below:
\[
H_{X(\fp)}\in\Coh (\sfS_\fp)=(\Coh \sfS)_\fp\subseteq\Coh\sfS.
\]
The equality is by Proposition~\ref{pr:loc-prime}. Thus $\gam_\fp
H_X\in\Coh\sfS$ for all $\fp$, by \eqref{eq:kos}. It follows from
Theorem~\ref{th:lgp_coh} that $H_X$ belongs to $\Coh\sfS$, and hence
that $X\in\sfS$.
\end{proof}

\begin{theorem}
\label{th:Hom}
For any pair of objects $X,Y$ in $\sfT$ the following holds:
\begin{align*}
  \Hom^*_\sfT(X,Y)=0\quad&\iff\quad \Hom^*_{\sfT_\fp}(X_\fp,Y_\fp)=0
  \text{ for all  }\fp\in\Spec R\\
  &\iff\quad \Hom^*_{\sfT_\fp}(X(\fp),Y(\fp))=0 \text{ for all
  }\fp\in\Spec R.
\end{align*}
\end{theorem}

\begin{proof}
  Let $\sfS=\Thick(X)$. Then Theorem~\ref{th:lgp_coh} yields the
  following equivalences:
\begin{align*}
  \Hom^*_\sfT(X,Y)=0\quad&\iff\quad H_Y\in\Coh \sfT/\sfS\\
    &\iff\quad H_{Y_\fp}\in (\Coh \sfT/\sfS)_\fp \text{ for all }\fp\in\Spec R\\
   &\iff\quad \gam_\fp H_{Y}\in (\Coh \sfT/\sfS)_\fp \text{ for all }\fp\in\Spec R.
\end{align*}
Using the identification $\Coh (\sfT_\fp/\sfS_\fp)=\Coh
(\sfT/\sfS)_\fp=(\Coh \sfT/\sfS)_\fp$ from Lemma~\ref{le:local-prime}
and the identity \eqref{eq:kos}, we obtain
\begin{align*}
  \Hom^*_\sfT(X,Y)=0\quad&\iff\quad \Hom^*_{\sfT_\fp}(X_\fp,Y_\fp)=0
  \text{ for all  }\fp\in\Spec R\\
  &\iff\quad \Hom^*_{\sfT_\fp}(X_\fp,Y(\fp))=0 \text{ for all
  }\fp\in\Spec R.
\end{align*}
In the last condition, $X_\fp$ can be replaced by $X(\fp)$. This
follows from the general fact that for any homogenous ideal $\fa$ of
$R$ and any pair of objects $U,V$ in $\sfT$
\[
\Hom_\sfT^*(U,V)=0\quad\iff \quad\Hom_\sfT^*(\kos U\fa,V)=0
\]
when $\Hom_\sfT^*(U,V)$ is $\mcV(\fa)$-torsion; for a proof use
\eqref{eq:koszul-object} or see \cite[Lemma~5.11]{\bik:2008a}.
\end{proof}

\section{Tensor triangulated categories}

Let $(\sfT,\otimes,\one)$ be a tensor triangulated category that is
essentially small.  The tensor product $\otimes\col
\sfT\times\sfT\to\sfT$ is then symmetric monoidal, exact in each
variable, and admits a unit $\one$.  The tensor product on $\sfT$
extends to a tensor product $\Coh\sfT\times\Coh\sfT\to\Coh\sfT$ that
we denote again by $\otimes$. We list the basic (and defining)
properties. For objects $X,Y\in\sfT$ and $F,G\in\Coh\sfT$ we have:
\begin{enumerate} 
\item $H_X\otimes H_Y\cong H_{X\otimes Y}$. 
\item $F\otimes -$ and $-\otimes G$  are exact and preserve filtered colimits.
\item $F\otimes G\cong G\otimes F$.
\end{enumerate}

\begin{lemma}
$(\Coh\sfT,\otimes,H_\one)$ is a  symmetric monoidal
category. \qed
\end{lemma}

\subsection*{Strongly monoidal functors}
A functor $f$ between symmetric monoidal categories is called
\emph{strongly monoidal} if there are isomorphisms
\[
\one\xra{\sim}f(\one)\qquad\text{and}\qquad  f(X)\otimes f(Y)\xra{\sim}f(X\otimes Y)
\]
that are natural and compatible with the monoidal structures. We have
the following projection formula.

\begin{lemma}
\label{le:projection}
Let $f\colon \sfT\to\sfU$ be a strongly monoidal exact functor between
tensor triangulated categories.  For $F\in\Coh\sfT$ and
$G\in\Coh\sfU$, there is a natural morphism
\[
\alpha_{F,G}\colon F\otimes f_* (G)\lto f_*(f^*(F)\otimes G).
\]
This is an isomorphism when $\sfT$ is generated as a triangulated category by $\one$.
\end{lemma}
\begin{proof}
  Observe that $f^*\colon\Coh\sfT\to\Coh\sfU$ is strongly
  monoidal. This is clear for representable functors; the general case
  follows by taking filtered colimits in one argument and then in the
  other. The morphism $\alpha_{F,G}$ is the adjoint of the composition
\[
  f^*(F\otimes f_*(G)) \xra{\sim}f^*(F)\otimes f^*f_*(G)\to
f^*(F)\otimes G.
\]
Observe that the objects $F$ such that $\alpha_{F,G}$ is an
isomorphism for all $G$ form a localising subcategory of $\Coh\sfT$
containing $H_\one$.  If $\one$ generates $\sfT$, then
$\Loc(H_\one)=\Coh\sfT$, by Lemma~\ref{le:Loc}. Thus $\alpha_{F,G}$ is
an isomorphism for all $F$ and $G$.
\end{proof}

\subsection*{Cohomological localisation}

Let $R$ be a noetherian graded commutative ring acting on $\sfT$. The
cohomological (co)localisation functors arising from this action can
be expressed as tensor functors.

\begin{proposition}\label{pr:tensor}
  Let $\mcV$ be a specialisation closed subset of $\Spec R$. Then
\[\gam_\mcV\cong\gam_\mcV H_\one\otimes-\qquad\text{and}\qquad 
L_\mcV\cong L_\mcV H_\one\otimes-.\]
\end{proposition}
\begin{proof}
  A simple calculation shows that one isomorphism implies the
  other. For instance, when $L_\mcV\cong L_\mcV H_\one\otimes-$ then
  $L_\mcV(\gam_\mcV H_\one\otimes-)=0$. This yields a morphism
  $\gam_\mcV H_\one\otimes-\to\gam_\mcV$ making the following diagram
  commutative.
\[\xymatrix{\cdots \ar[r]& \gam_\mcV H_\one\otimes-\ar[r]\ar[d]&\Id \ar[r]\ar@{=}[d]
&L_\mcV H_\one\otimes-\ar[r]\ar[d]&\cdots\\
\cdots\ar[r] &\gam_\mcV\ar[r]&\Id\ar[r]&L_\mcV\ar[r]&\cdots}\]
The five lemma then shows that this is an isomorphism.

It follows from the description of $L_{\mcV(r)}$ in
Lemma~\ref{le:gamma_r_colim} that the assertion holds for a closed set
$\mcV(r)$ given by some $r\in R$. Lemma~\ref{le:composite} then
implies the assertion for a closed set $\mcV(\fa)$ given by a finitely
generated ideal $\fa$ of $R$, and Lemma~\ref{le:gamma_colim} implies
the assertion for an arbitrary specialisation closed subset.
\end{proof}

\section{Stratification}

Let $R$ be a noetherian graded commutative ring and $\sfT$ be an
essentially small $R$-linear triangulated category.  In this section
we study the stratification of $\sfT$ and $\Coh\sfT$; this is the
analogue of stratification for compactly generated triangulated
categories introduced in \cite[\S4]{\bik:2011a} and inspired by
\cite[\S6]{ Hovey/Palmieri/Strickland:1997a}.

\subsection*{Stratification}
The triangulated category $\sfT$ is called \emph{minimal} if $\sfT$
admits no proper thick subcategory. This means if $\sfS\subseteq\sfT$
is a thick subcategory then $\sfS = 0$ or $\sfS = \sfT$. Analogously,
$\Coh\sfT$ is said to be minimal if $\Coh\sfT$ admits no proper
localising subcategory. Clearly, $\sfT$ is minimal when $\Coh\sfT$ is
minimal.

\begin{definition}
  We say that $\sfT$ is \emph{stratified} by the action of $R$ if
  $\gam_\fp\sfT$ is minimal for each $\fp\in\Spec R$. In the same
  vein, $\Coh\sfT$ is \emph{stratified} by the action of $R$ if
  $\Coh\gam_\fp\sfT$ is minimal for each $\fp\in\Spec R$.
\end{definition}
In each case stratification yields a classification of thick or
localising subcategories in terms of subsets of $\Spec R$; see
Corollary~\ref{co:classification}.

\begin{remark}
  Suppose that $\sfT$ is minimal. Then $\sfT$ is stratified by any
  $R$-action, and in particular, by the canonical action of
  $\bbZ$. Moreover, there is a unique prime $\fp\in\Spec R$ such that
  $\gam_\fp\sfT\neq 0$. Clearly, this implies $\gam_\fq\sfT=0$ for all
  $\fq\neq \fp$ in $\Spec R$.
\end{remark}

\subsection*{Consequences of stratification} 
It is convenient to set $\supp_R X=\supp_R H_X$ for each object
$X\in\sfT$. Observe that \eqref{eq:kos} implies
\[\supp_R X=\{\fp\in\Spec R\mid X(\fp)\neq 0\},\]
and this can be reformulated in terms of the following identity which is an
immediate consequence of Lemma~\ref{le:kos}
\begin{equation}\label{eq:stratification}
\Thick(X(\fp))=\Thick(X_\fp)\cap\gam_\fp\sfT=\gam_\fp\Thick(X).
\end{equation}

\begin{theorem}
\label{th:stratification}
Suppose that $\sfT$ is stratified by the action of $R$. Given objects
$X,Y$ in $\sfT$, we have
\begin{align*}
X\in\Thick(Y) \quad&\iff\quad \supp_R X\subseteq\supp_R Y,\\
\Hom^*_\sfT(X,Y)=0\quad&\iff\quad (\supp_R X)\cap(\supp_R
Y)=\varnothing.
\end{align*}
\end{theorem}

\begin{proof}
  For the first assertion set $\sfS=\Thick(Y)$. The local-global
  principle from Theorem~\ref{th:lgp} gives the first equivalence:
\begin{align*}
  X\in\Thick(Y) \quad&\iff\quad X(\fp)\in\gam_\fp\sfS\text{ for all  }\fp\in\Spec R\\
  &\iff\quad \supp_R X\subseteq\supp_R Y.
\end{align*}
The second equivalence uses the minimality of $\gam_\fp\sfT$ and the
identity \eqref{eq:stratification}.

For the second assertion recall from Theorem~\ref{th:Hom} that
\begin{align*}
  \Hom^*_\sfT(X,Y)=0\quad\iff\quad \Hom^*_{\sfT_\fp}(X(\fp),Y(\fp))=0
  \text{ for all }\fp\in\Spec R.
\end{align*}
The minimality of $\gam_\fp\sfT$ implies for objects $U,V$ in
$\gam_\fp\sfT$ that $\Hom^*_{\sfT_\fp}(U,V)\neq 0$ iff $U\neq 0\neq
V$. Thus $ \Hom^*_{\sfT_\fp}(X(\fp),Y(\fp))=0$ iff $\fp\not\in
(\supp_R X)\cap(\supp_R Y)$.
\end{proof}

Theorem~\ref{th:stratification} has a converse when endomorphism rings
are finitely generated.

\begin{proposition}
\label{pr:stratification-converse}  
Suppose that for each object $X$ in $\sfT$ the endomorphism ring
$\End^*_\sfT(X)$ is finitely generated over $R$. If $\sfT$ is not
stratified by $R$, then there are objects $X,Y\in\sfT$ such that
$\supp_R X=\supp_R Y$ but $\Thick(X)\neq\Thick(Y)$.
\end{proposition}

\begin{proof}
  Assume $\gam_\fp\sfT$ is not minimal. Thus there are non-zero objects
  $X_\fp,Y_\fp$ in $\gam_\fp\sfT$ such that
  $\Thick(X_\fp)\neq\Thick(Y_\fp)$. In $\sfT$ consider the objects
  $X'=\kos X\fp$ and $Y'=\kos Y\fp$.  Then $\supp_R
  X'=\mcV(\fp)=\supp_R Y'$ by Proposition~\ref{pr:tria_support}. On
  the other hand, Lemma~\ref{le:kos} gives the equalities below:
\[
\Thick(X'_\fp)=\Thick(X_\fp)\neq\Thick(Y_\fp)=\Thick(Y'_\fp),
\]
so that $\Thick(X')\neq\Thick(Y')$.
\end{proof}

The Hom vanishing statement in Theorem~\ref{th:stratification} can be
strengthened when morphism spaces are finitely generated over $R$.
For an $R$-module $M$, we write
\[
\Supp_R M=\{\fp\in\Spec R\mid M_\fp\neq 0\}.
\] 
The following theorem can be used to explain results on the symmetry
of Hom vanishing, as studied in \cite{Avramov/Buchweitz:2000a,
  Bergh/Oppermann:2011a}.

\begin{theorem}
Let $X$ and $Y$ be objects in $\sfT$.
\begin{enumerate}
\item If $\Hom_\sfT^*(X,Y)$ is finitely generated over $R$, then
\[
\Supp_R  \Hom_\sfT^*(X,Y)\subseteq (\supp_R X)\cap(\supp_R Y).
\]
\item If  $\sfT$ is stratified by the action of $R$, then
\[
\Supp_R  \Hom_\sfT^*(X,Y)\supseteq (\supp_R X)\cap(\supp_R Y).
\]
\end{enumerate}
\end{theorem}

\begin{proof}
  (1) Let $\fp\in \Supp_R \Hom_\sfT^*(X,Y)$. Thus
  $\Hom^*_{\sfT_\fp}(X_\fp,Y_\fp)\neq 0$. The assumption implies that
  this is finitely generated, and an application of Nakayama's lemma
  gives
\[
\Hom^*_{\sfT_\fp}(X(\fp),Y(\fp))\neq 0;
\] 
use \eqref{eq:koszul-object} or see \cite[Lemma~5.11]{\bik:2008a} for
a proof. Thus $\fp \in (\supp_R X)\cap(\supp_R Y)$.

(2) Let $\fp \in (\supp_R X)\cap(\supp_R Y)$. Then stratification
implies \[\Hom^*_{\sfT_\fp}(X(\fp),Y(\fp))\neq 0,\] and therefore
\[
\Hom^*_\sfT(X,Y)_\fp\cong \Hom^*_{\sfT_\fp}(X_\fp,Y_\fp)\neq 0.
\]
Thus $\fp\in \Supp_R  \Hom_\sfT^*(X,Y)$.
\end{proof}

\subsection*{Perfect complexes}

Let $A$ be a noetherian commutative ring. We denote by $\sfD(A)$ the
derived category of the category of $A$-modules. An object in
$\sfD(A)$ is called \emph{perfect} if it is isomorphic to a bounded
complex of finitely generated projective $A$-modules; these form a
thick subcategory denoted by $\sfD^\per(A)$. For $X\in\sfD(A)$ set
\[H_X=\Hom_A(-,X)|_{\sfD^\per(A)}.\]

The ring $R=A$ acts canonically on $\sfT=\sfD^\per(A)$ and we show
that $\Coh\sfT$ is stratified by this action. The residue fields play
a special role. For $\fp\in\Spec R$ let $k(\fp)=A_\fp/\fp_\fp$, viewed
as a complex concentrated in degree zero.

\begin{lemma}\label{le:perfect}
Let  $\fp\in\Spec R$. Then $\Coh\gam_\fp\sfT=\Loc(H_{k(\fp)})$.
\end {lemma} 

\begin{proof}
  Since $H_{k(\fp)}$ is $\fp$-local and $\mcV(\fp)$-torsion when
  evaluated at any object from $\sfT$, it belongs to
  $\Coh\gam_\fp\sfT$; this justfies one inclusion. For the other one,
  it is convenient to identify $\sfT_\fp$ and $\sfD^\per(A_\fp)$. Thus
  an object $X$ in $\gam_\fp\sfT$ is a perfect complex over $A_\fp$
  such that its cohomology is of finite length over $A_\fp$. It
  follows that
 \[
  X\in\Thick(k(\fp))\subseteq\sfD(A_\fp).
 \] 
 This is easily shown by an induction on the number of integers $n$
 such $H^n(X)\neq 0$; see for example
 \cite[Example~3.5]{Dwyer/Greenlees/Iyengar:2006}. Thus $H_X\in
 \Loc(H_{k(\fp)})$.
\end{proof}

\begin{theorem}\label{th:ring-stratification}
Let $A$ be a commutative noetherian ring. Then $\Coh\sfD^\per(A)$ is
stratified by the canonical action of $A$.
\end{theorem}

\begin{proof}
  Fix $\fp\in\Spec A$.  We need to show that $\Coh\gam_\fp\sfT$ is
  minimal, which is equivalent to $\Loc(F)=\Loc(H_{k(\fp)})$ for each
  non-zero $F\in \Coh\gam_\fp\sfT$, by Lemma~\ref{le:perfect}. This is
  clear for $F=\gam_\fp H_A$.  This gives the second of the following
  equalities:
\[
\Loc(F)=\Loc(F\otimes \gam_\fp H_A)=\Loc(F\otimes H_{k(\fp)}).
\] 
The first one is by Proposition~\ref{pr:tensor}.  Let $f$ denote the functor 
\[
-\otimes _A k(\fp)\colon \sfD^\per(A)\lto \sfD^\per(k(\fp)).
\] 
Then $F\otimes H_{k(\fp)}\cong f_*f^*(F)$ by
Lemma~\ref{le:projection}.  Thus $F\otimes H_{k(\fp)}$ is a direct sum
of suspensions of $H_{k(\fp)}$, since every object in
$\sfD^\per(k(\fp))$ is a direct sum of suspensions of $k(\fp)$.  It
follows that $\Loc(F)=\Loc(H_{k(\fp)})$.
\end{proof}

The following result is due to Hopkins \cite{Hopkins:1987a} and Neeman
\cite{Neeman:1992b}.

\begin{corollary}
Let $\sfS\subseteq \sfD^\per(A)$ be a thick subcategory. Then
\[\sfS=\{X\in \sfD^\per(A)\mid \supp_A X\subseteq\mcV\}\] for some 
specialisation closed subset $\mcV\subseteq\Spec A$.
\end{corollary}
\begin{proof}
  The assertion follows from Theorems~\ref{th:stratification} and
  \ref{th:ring-stratification}, using the fact that $\supp_A X$ is
  specialisation closed for each $X\in\sfD^\per(A)$ by
  Proposition~\ref{pr:tria_support}.
\end{proof}

\begin{remark}
  Theorem~\ref{th:ring-stratification} generalises with same proof in
  two directions as follows.

  (1) Let $A$ be a commutative differential graded algebra such that
  the ring $H^*(A)$ is noetherian. If $A$ is formal, then
  $\Coh\sfD^\per(A)$ is stratified by the canonical
  $H^*(A)$-action. This is an analogue of Theorem~8.1 in
  \cite{Benson/Iyengar/Krause:2011a} that asserts that $\sfD(A)$ is
  stratified by $H^{*}(A)$.

  (2) Let $A$ be a graded commutative noetherian ring. More precisely, we fix
  an abelian grading group $G$ endowed with a symmetric bilinear form
  \[(-,-)\colon G\times G\lto\bbZ/2,\] and $A$ admits a
  decomposition \[A=\bigoplus_{g\in G}A_g\] such that the
  multiplication satisfies $A_g A_h\subseteq A_{g+h}$ for all $g,h\in
  G$ and $xy=(-1)^{(g,h)}yx$ for all homogeneous elements $x\in A_g$,
  $y\in A_h$. We consider $G$-graded $A$-modules with degree zero
  morphisms and let $\sfD(A)$ denote its derived category. Localising
  subcategories of $\Coh\sfD^\per(A)$ and $\sfD(A)$ are supposed to be
  closed under twists, where the \emph{twist} of a module or complex
  $X$ by $g\in G$ is given by $X(g)_h= X_{g+h}$.  Suitably adapting
  definitions and constructions to take into account twists, one can
  establish that $\Coh\sfD^\per(A)$ is stratified by the canonical
  $A$-action. This is an analogue of Corollary~5.7 in \cite{DS:2013a}
  that asserts that $\sfD(A)$ is stratified by $A$.
\end{remark}

\section{Compactly generated triangulated categories}

Let $R$ be a noetherian graded commutative ring and $\sfT$ be a
compactly generated $R$-linear triangulated category. The subcategory
of compacts, $\sfT^{\sfc}$, is an essentially small triangulated
category and has an induced $R$-action. In
\cite{Benson/Iyengar/Krause:2008a} we developed a theory of local
cohomology and support for $\sfT$. In this section, we use the
\emph{restricted Yoneda functor}
\[
\sfT\lto\Coh\sfT^{\sfc},\quad X\mapsto H_X=\Hom_{\sfT}(-,X)|_{\sfT^{\sfc}},
\]
to compare it with the one for $\Coh \sfT^{\sfc}$ introduced in this article.

\begin{remark}
  The morphisms annihilated by the functor $\sfT\to\Coh\sfT^{\sfc}$
  are called \emph{phantom maps}.  In the context of the stable module
  category $\sfT=\StMod kG$ of a finite group $G$, these were studied
  by Benson and Gnacadja. In particular, in
  \cite[\S4]{Benson/Gnacadja:2001a} it is shown that there are
  filtered systems in $\sfT^{\sfc}=\stmod kG$ that do not lift to
  $\mod kG$. As a consequence, there are objects in $\Coh\sfT^{\sfc}$
  that are not in the image of $\sfT\to\Coh\sfT^{\sfc}$, namely the
  filtered colimit of the corresponding representable functors.  In
  the context of the derived category of a commutative noetherian
  ring, examples of filtered systems that do not lift can be found in
  Neeman \cite{Neeman:1997a}.
\end{remark}

\subsection*{Cohomological localisation}

Given a specialisation closed subset $\mcV$ of $\Spec R$, there is an
exact localisation functor $\bar L_\mcV\colon\sfT\to\sfT$ such
that $\bar L_\mcV X=0$ iff $H_X$ is $\mcV$-torsion; see
\cite[\S4]{\bik:2008a}. The corresponding colocalisation functor is
denoted by $\bar\gam_\mcV$. Thus each $X\in\sfT$ fits into an
exact \emph{localisation triangle}
\[\bar\gam_\mcV X\lto X\lto\bar L_\mcV X\lto.\]

The following proposition says that notions developed in
\cite{\bik:2008a} for compactly generated triangulated categories are
determined by analogous concepts for the category of cohomological
functors which are  discussed in this work. This applies, for
instance, to the notion of support.

\begin{proposition}
Let $\mcV\subseteq\Spec R$ be specialisation closed and $X\in\sfT$. Then 
\[
H_{\bar \gam_\mcV X}\cong \gam_\mcV H_X\qquad\text{and}\qquad H_{\bar L_\mcV X}\cong L_\mcV H_X.
\]
\end{proposition}

\begin{proof}
It follows from Proposition~\ref{pr:loc-seq} that the long exact sequence
\[
\cdots\lto H_{\Si^{-1}(\bar L_\mcV X)}\lto H_{\bar\gam_\mcV X}\lto
H_X\lto H_{\bar L_\mcV X}\lto H_{\Si (\bar\gam_\mcV X)}\lto\cdots
\]
is isomorphic to the localisation sequence \eqref{eq:loc-seq} for
$\sfT^{\sfc}_\mcV\subseteq\sfT^{\sfc}$, applied to $H_X$.
\end{proof}

\subsection*{Localising subcategories}

A full triangulated subcategory of $\sfT$ is \emph{localising} if it
is closed under forming coproducts. Following \cite[\S3]{\bik:2011a}, we
say that the \emph{local-global principle} holds for $\sfT$, if for
each object $X\in\sfT$ we have
\[\Loc(X)=\Loc(\{\bar\gam_\fp X\mid \fp\in\Spec R
\}).\]
This local-global principle has been established in a number of
relevant cases. For instance, it holds when $R$ has
finite Krull dimension   \cite[Corollary~3.5]{\bik:2011a}, or when
$\sfT$ admits a model \cite[Theorem~6.9]{Stevenson:2011a}.

We obtain an alternative proof of Theorem~\ref{th:lgp}, provided the
local-global principle holds for $\sfT$.

\begin{proposition}
\label{pr:big-implies-small}  
The local-global principle for $\sfT$ implies the principle for $\sfT^{\sfc}$.
\end{proposition}

\begin{proof}
  We verify that conditions (1)--(3) of Theorem~\ref{th:lgp} are
  equivalent. Evidently (1) $\Ra$ (2) and (2) $\Ra$ (3).

  Assume (3) holds. Given $\fp\in\Spec R$, it follows from
  \cite[Theorem~3.1]{\bik:2011a} that $\sfS_{\fp}$ is contained in
  $\Loc(\sfS)$. Thus
\[
\Loc_{\sfT}(\bar\gam_{\fp}X) = \Loc(X(\fp)) \subseteq \Loc(\sfS_\fp)\subseteq \Loc(\sfS),
\]
where the equality holds by \cite[Lemma~3.8]{\bik:2011a}.  Now the
local-global principle for $\sfT$ yields that $X$ belongs to
$\Loc(\sfS)$. It remains to observe that this implies $X\in\sfS$,
because $X$ is compact and $\sfS$ is a subcategory of compact objects.
\end{proof}

\subsection*{Cohomological localising subcategories}

Any localising subcategory of $\Coh\sfT^{\sfc}$ induces one of $\sfT$
via the restricted Yoneda functor. However, we do not know if this is
a bijection between the corresponding localising subcategories. This
changes when one restricts to localising subcategories that are
defined cohomologically.

We call a localising subcategory $\sfS\subseteq\sfT$
\emph{cohomological} if there is a  cohomological functor
$F\colon\sfT\to\sfA$ such that 
\begin{enumerate}
\item $\sfA$ is an abelian category with exact filtered colimits,
\item $F$ preserves coproducts, and
\item $\sfS$ equals the  full subcategory of objects in
  $\sfT$  annihilated by $F$.
\end{enumerate}
Analogously, a localising subcategory $\sfC\subseteq\Coh\sfT^{\sfc}$ is called
\emph{cohomological} if there is an exact functor
$F\colon\Coh\sfT^{\sfc}\to\sfA$ such that 
\begin{enumerate}
\item $\sfA$ is an abelian category with exact filtered colimits,
\item $F$ preserves coproducts, and
\item $\sfC$ equals the  full subcategory of objects in
  $\Coh \sfT^{\sfc}$  annihilated by $F$.
\end{enumerate}

\begin{example}
  An intersection of cohomological localising subcategories is
  cohomological. Given a subset $\mcU\subseteq\Spec R$, the localising
  subcategories
\[
\{X\in\sfT\mid \supp_R X\subseteq\mcU\}\qquad\text{and}\qquad \{X\in\Coh\sfT^{\sfc}\mid \supp_R X\subseteq\mcU\}
\]
are cohomological.
\end{example}

\begin{proposition}
Taking a localising subcategory  $\sfC\subseteq\Coh\sfT^{\sfc}$ to
$\{X\in\sfT\mid H_X\in\sfC\}$ induces a bijection between
\begin{enumerate}
\item[--] the cohomological localising subcategories of $\Coh\sfT^{\sfc}$, and
\item[--] the cohomological localising subcategories of $\sfT$.
\end{enumerate}
\end{proposition}
\begin{proof}
  To describe the inverse map, let $F\colon\sfT\to\sfA$ be a
  cohomological functor which preserves coproducts. This extends
  essentially uniquely to an exact and coproduct preserving functor
  $\bar F\colon\Coh\sfT^{\sfc}\to\sfA$ by sending a filtered colimit of
  representable functors $\colim_\alpha H_{X_\alpha}$ to
  $\colim_\alpha F(X_\alpha)$. It remains to observe that for each
  $X\in\sfT$ we have $F(X)=0$ iff $\bar F(H_X)=0$. Thus the localising
  subcategories determined by $F$ and $\bar F$ correspond to each
  other under the above assignment.
\end{proof}

\subsection*{Acknowledgements} It is our pleasure to thank an
anonymous referee  for a critical reading and many helpful comments.

\end{document}